%% file: arxiv.tex
\documentclass[11pt,english]{article}
\usepackage{a4wide}
\usepackage{threeparttable}
\usepackage{booktabs}
\usepackage{threeparttable}
\usepackage{changepage}
\usepackage{multirow}
\usepackage{caption}
\usepackage{algorithm}
\usepackage{algorithmic}
\newcommand{\twolines}[2]{%
  $\begin{gathered}
  \text{#1}\\[-1.15ex]
  \text{#2}
  \end{gathered}$%
}
\usepackage{xcolor}
\usepackage{pifont}
\newcommand{\cmark}{\textcolor{green!60!black}{\ding{51}}}
\newcommand{\xmark}{\textcolor{red!90!black}{\ding{55}}}

\usepackage{array}
\usepackage{tabularx}
\newcolumntype{C}{>{\centering\arraybackslash}X}

\usepackage[round]{natbib}

\usepackage[backref=page]{hyperref}
\renewcommand*{\backref}[1]{}
\renewcommand*{\backrefalt}[4]{%
    \ifcase #1 %
    \or        (Cited on page~#2.)%
    \else      (Cited on pages~#2.)%
    \fi}

\usepackage{nicefrac}
\usepackage{authblk}

\input{my_macros}

\title{Decentralized SGD under Heavy-Tailed Noise: Optimal Convergence Rates and the Role of Gradient Clipping}
\author{Aleksandar Armacki, Haoyuan Cai, Ali H. Sayed}
\affil{\'Ecole Polytechnique F\'ed\'erale de Lausanne, Lausanne, Switzerland, \\
\texttt{\{aleksandar.armacki, haoyuan.cai, ali.sayed\}@epfl.ch}
}
\date{}

\begin{document}

\maketitle

\begin{abstract}
  Heavy-tailed noise has been widely observed in modern machine learning, motivating the use of methods like gradient clipping and normalization. While these methods are well understood in centralized settings, much less is known in decentralized ones, where applying a nonlinearity to local gradients affects both optimization and consensus. Recent works on decentralized non-convex optimization have studied both clipping and normalization under heavy-tailed noise, with clipping yielding suboptimal rates and normalization needing local momentum or mini-batches to converge. This raises the question: \emph{can a baseline decentralized method using a nonlinearity achieve optimal convergence rates under heavy-tailed noise?} We answer affirmatively with clipped decentralized SGD (\dsgd). For smooth non-convex costs under bounded $p$-th moment noise, $p \in (1,2]$, we show that clipped \dsgd achieves order-optimal rates both with high probability and in expectation. Moreover, we establish a linear speed-up in the number of agents, which, to our knowledge, has not been shown for decentralized methods with clipping.
The key technical ingredient is a sharp analysis of the consensus gap that exploits the structure of clipping, relegating network effects to higher-order terms. Our results highlight an important distinction between clipping and normalization in decentralized settings: while normalized \dsgd can fail to converge, clipping retains magnitude information, enabling \dsgd to be convergent and order-optimal. Numerical experiments validate our theory.
\end{abstract}

\section{Introduction}

Modern learning systems, fueled by an abundance of data, often incur huge computation and storage costs. This naturally gives rise to distributed training, a paradigm where $n$ agents jointly train a model, storing and processing their data locally, only exchanging model or gradient parameters \citep{sayed-networks,pmlr-v54-mcmahan17a,vlaski_et_al}. Two well-known instances are the client-server and the decentralized setup, differing mainly in the communication is performed: in the client-server setup agents communicate with a central server, while in the decentralized one agents communicate with each other. Noting that the client-server setup is mathematically equivalent to the decentralized one over a fully connected network, we focus on the more general decentralized setup.

Another issue hampering modern learning systems is the emergence of heavy-tailed noise. First observed by \cite{zhang2020adaptive} during training of the BERT model, the effects of heavy-tailed noise have been extensively studied in centralized settings, e.g., \citep{simsekli2019tail,heavy-tail-phenomena,barsbey-heavy_tails_and_compressibility,heavy-tail-book,pmlr-v238-battash24a}, with nonlinear methods like clipping, normalization and sign widely used to deal with heavy tails. Recently, \cite{kun2025heavy} showed that heavy-tailed noise arises naturally in decentralized settings, with decentralized SGD (\dsgd) in some cases exhibiting heavier tails than its centralized counterpart. However, as we discuss next, convergence under heavy-tailed noise and the role of nonlinear methods are yet to be fully understood in decentralized settings.

\begin{table*}[t]
\captionsetup{font=scriptsize}
\caption{Comparison of decentralized methods using gradient nonlinearities. ``Algorithm'' refers to the underlying algorithmic dynamics, e.g., ``GT + mini-b.'' indicates gradient tracking and mini-batch gradients. ``Nonlinearity'' specifies the nonlinear map applied to the gradient estimator. ``Noise Moment'' gives the moment assumption, with $p=2$ implying bounded variance and $p < 2$ allowing heavier tails. ``Guarantees'' indicates high-probability (HP) or in-expectation (IE) results. ``Oracle Comp.'' gives the per-agent number of \sfo calls needed to reach an $\epsilon$-stationary point.$^\S$ ``Lin. Speed-up'' indicates whether linear speed-up in the number of agents is achieved.}
\label{tab:comp}

\centering
\begin{threeparttable}

{\scriptsize
\begin{sc}

\renewcommand{\tabularxcolumn}[1]{m{#1}}
\renewcommand{\arraystretch}{1.15}

\begin{tabularx}{\textwidth}{CCCCCCC}
\toprule
Work
& Algorithm
& Nonlinearity
& Noise Moment
& Guarantees
& Oracle Comp.
& Lin. Speed-up
\\
\midrule

\cite{clipped-dsgd1}
& \dsgd
& clipping
& $p \in (1,2]$
& HP
& $\wbigO\bigg(\epsilon^{\frac{4p}{1-p}}\bigg)$
& \xmark
\\
\midrule

\cite{li2025porter}$^\dagger$
& gt + mini-b. + compr. + ef
& \twolines{smooth}{clipping}
& $p = 2$
& IE
& $\bigO\Big(\epsilon^{-4}\Big)$
& \xmark
\\
\midrule

\cite{yu2026decentralized}
& gt + local momentum
& normalization
& $p \in (1,2]$
& IE
& $\bigO\bigg(\epsilon^{\frac{3p-2}{1-p}}\bigg)$
& \xmark
\\
\midrule

\cite{luoluo2026near}$^\ddagger$
& gt + mini-b.
  + acc. gossip
& normalization
& $p \in (1,2]$
& IE
& $\bigO\bigg(\epsilon^{\frac{3p-2}{1-p}} / n\bigg)$
& \cmark
\\
\midrule

This work
& \dsgd
& clipping
& $p \in (1,2]$
& HP \& IE
& $\wbigO\bigg(\epsilon^{\frac{3p-2}{1-p}} / n\bigg)$
& \cmark 
\\
\bottomrule
\end{tabularx}

\end{sc}
}

{\scriptsize
\begin{tablenotes}
    \item[$\S$] Because prior works use different stationarity metrics, we compare oracle complexities using target criteria based on both network-averaged iterates and agent-averaged local iterates, i.e., $\min_{t \in [T]}\|\nabla f(\oxt)\|^q \leq \epsilon^q$ and $\min_{t \in [T]}\big[\frac{1}{n}\sum_{i \in [n]}\|\nabla f(\xit)\|^q\big] \leq \epsilon^q$, $q \in \{1,2\}$.
    \item[$\dagger$] \cite{li2025porter} use communication compression and error-feedback (EF) techniques, with mini-batch gradients of size $b \propto \epsilon^{-2}$. The smooth clipping operator is defined as $\varphi(x) = \frac{\tau}{\tau + \|x\|}x$, for a user-specified $\tau > 0$. 
    \item[$\ddagger$] To mitigate network effects and achieve linear speed-up, \cite{luoluo2026near} use mini-batches of size $b \propto \epsilon^{\frac{p}{1-p}}/n$ and, for undirected graphs, perform $\widetilde{\Theta}\big(1/\sqrt{1 - \lambda}\big)$ communication rounds per iteration, where $1 - \lambda$ is the network spectral gap (see Section \ref{sec:main} for definition).
\end{tablenotes}
}

\end{threeparttable}
\end{table*}

\subsection{Related Literature}

We next review convergence results under heavy-tailed noise, focusing on non-convex costs. For further literature, including convex results, see Supplement \ref{sup:lit}.  

\textbf{Centralized Settings.} Moment assumptions on the stochastic first-order oracle (\sfo) including heavy-tailed regimes with only a finite $p$-th moment, $p > 1$, go back to \cite{nemirovski1983problem}. More recently, \cite{zhang2020adaptive} noticed that the gradient noise during training of attention models resembles Levy $\alpha$-stable noise and proposed the use of the bounded $p$-th moment condition. They establish an in-expectation (IE) lower bound $\Omega\Big(\epsilon^{\frac{3p-2}{1-p}}\Big)$ on the oracle complexity,\footnote{Meaning the number of \sfo calls needed to reach an $\epsilon$-stationary point, see Table \ref{tab:comp} for details.} and show it is achieved by clipped SGD. \cite{cutkosky2021high} complement this in the high-probability (HP) sense, for clipped + normalized momentum SGD. \cite{sadiev2023highprobability,nguyen2023improved} then establish HP guarantees for clipped SGD, with \cite{nguyen2023improved} showing the optimal HP oracle complexity, while the same is established for normalization in \citep{liu2025nonconvex,hubler2025normalization} and sign SGD in \citep{kornilov2025sign}. \cite{liu2023breaking} show an accelerated rate and improved complexity, using a STORM-like method \citep{cutkosky2019momentum} with clipping + normalization, while \cite{armacki2026tight} demonstrate order of magnitude faster tail decay of clipped SGD. A related line of work studies symmetric heavy-tailed noise, dating back to \cite{polyak-adaptive-estimation}. Among them, \cite{bernstein2019signsgd_iclr,chen2020understanding} analyze IE guarantees of sign and clipped SGD, showing $\mathcal{O}(\epsilon^{-4})$ oracle complexity, while \cite{armacki2025high,armacki2024_ldp+mse,armacki2026sharp} study a nonlinear SGD framework that includes sign, normalization and clipping, among other results establishing $\bigO(\epsilon^{-4})$ IE and HP oracle complexity. Finally, a recent line of work shows that vanilla SGD can converge under heavy-tailed noise without any nonlinear methods, e.g., \citep{fatkhullin2025can,liu2026can}. 

\textbf{Decentralized Settings.} Nonlinear methods have a long history of use in decentralized settings, including applications such as estimation, adversarial learning and differential privacy, e.g., \citep{kar_estimaton,chen2019resilient,yu-secure,zhu2025dp-lin}. However, only a few works study convergence of nonlinear decentralized methods under heavy-tailed noise. In particular, \cite{clipped-dsgd1} study HP convergence of clipped \dsgd in online dynamic optimization under $p$-th moment noise, requiring bounded gradients and achieving $\wbigO\Big(\epsilon^{\frac{4p}{1-p}}\Big)$ per-agent oracle complexity.\footnote{Per-agent complexity is the number of \sfo calls needed per agent to reach an $\epsilon$-stationary point, implying a total complexity of $n$ $\times$ the per-agent one.} \cite{li2025porter} study communication-efficient learning under bounded variance noise ($p = 2$), requiring a strong heterogeneity bound (see Subsection \ref{subsec:assump} and Footnote \ref{fnote:het} ahead) and using a complex method (see Table \ref{tab:comp}) to achieve $\bigO\big(\epsilon^{-4}\big)$ IE oracle complexity. \cite{yu2026decentralized} study convergence of gradient tracking (GT) with normalization and local momentum under $p$-th moment noise, achieving $\bigO\Big(\epsilon^{\frac{3p-2}{1-p}}\Big)$ IE oracle complexity, while \cite{luoluo2026near} similarly study GT with normalization, using large mini-batches and accelerated gossip (i.e., multiple communication rounds per iteration) to achieve the optimal $\bigO\Big(\epsilon^{\frac{3p-2}{1-p}} / n\Big)$ per-agent IE oracle complexity.\footnote{Optimal in the sense that the total complexity matches the optimal complexity achieved in centralized settings.} All of the said works require either strong heterogeneity bounds, or algorithmic modifications like local-momentum, mini-batches and accelerated gossip, with only \citep{luoluo2026near} achieving the optimal complexity and linear speed-up in the number of agents.\footnote{While the $\bigO\big(\epsilon^{-4}\big)$ per-agent complexity in \cite{li2025porter} matches the optimal centralized complexity under noise with bounded variance, it implies a sub-optimal $\bigO\big(n\epsilon^{-4}\big)$ total complexity. Similarly, while the per-agent complexity in \citep{yu2026decentralized} matches the optimal centralized complexity, the total complexity is sub-optimal by a factor of $n$, due to not achieving linear speed-up.} The need for algorithmic modifications when using nonlinear methods is further supported in \citep[Proposition 1]{yu2026decentralized}, where it is shown that normalized \dsgd with noiseless gradients can fail to converge even on a simple quadratic cost. These observations naturally led us to ask:
\begin{center}
    \emph{Can a standard decentralized method with a nonlinear map achieve optimal complexity under heavy-tailed noise, without further algorithmic modification?}
\end{center}

To answer this question, we revisit the guarantees of \dsgd using gradient clipping. In Figure \ref{fig:normalization_clipping_example} we show that clipping indeed fixes the non-convergence induced by normalization, while Table \ref{tab:comp} provides a detailed comparison of nonlinear decentralized methods that converge to stationarity, giving a positive answer to the question above, in the form of clipped \dsgd.

\begin{figure}[t]
\centering
\includegraphics[width=0.5\linewidth]{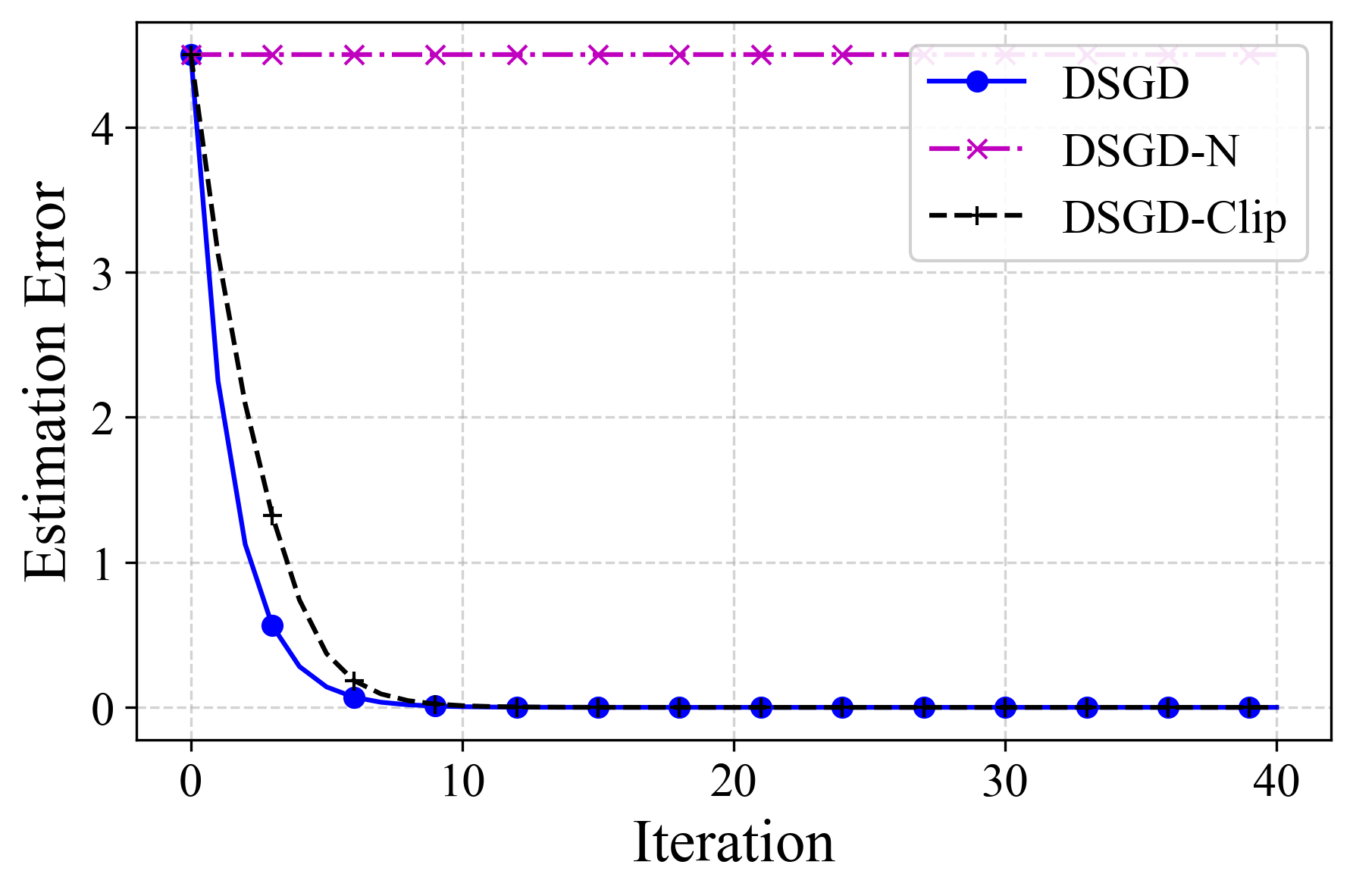}
\caption{Convergence of vanilla, normalized and clipped \dsgd on a quadratic cost. While normalization gets stuck at the initial point and fails to converge, both vanilla and clipped \dsgd converge to the true minima. For more details, see Supplement \ref{sup:num}.}
\label{fig:normalization_clipping_example}
\end{figure}

\subsection{Contributions} 

Our full contributions can be summarized as follows.

\begin{itemize}[leftmargin=*]
    \item We study convergence of decentralized methods using a nonlinearity under heavy-tailed noise, showing that baseline \dsgd with gradient clipping achieves \emph{optimal oracle complexity} and \emph{linear speed-up}, under standard assumptions (see Subsection \ref{subsec:assump}) and without any additional algorithmic modifications.

    \item Compared to results on clipping in decentralized settings, we relax the strong heterogeneity bounds and, to our knowledge, show the first linear speed-up result for clipped \dsgd. Compared to works using normalization, we show that a simple decentralized method can achieve optimal complexity with an appropriately chosen nonlinearity, establishing important insights on the effects of clipping and normalization in decentralized settings (see Subsection \ref{subsec:guarantees}).

    \item We extensively test our theory on several problems, including decentralized transformer training on real data, showing that: $(i)$ clipping improves performance under heavy tails, $(ii)$ clipped \dsgd achieves linear speed-up; $(iii)$ clipped \dsgd regularly performs on par with or outperforms normalized \dsgd. 
\end{itemize}

\textbf{Novelty.} Our results are facilitated by several novelties and improvements. First, we relax the strong heterogeneity bounds in prior works via a sharp analysis and careful choice of clipping radius. Next, a novel consensus bound in Lemma \ref{lm:consensus-bound} and a novel coupling of step-size and clipping radius (see Subsection \ref{subsec:guarantees}) ensure that the network only affects higher-order terms. Finally, we achieve linear speed-up by combining these results, an improved analysis, improved choice of step-size and clipping radius, and Lemma \ref{lm:noise-behavior}, which shows that averaging clipped vectors reduces the variance. 

\textbf{Paper Organization.} The rest of the paper is organized as follows. Section \ref{sec:prelim} gives the preliminaries, Section \ref{sec:main} states the main results, Section \ref{sec:num} presents experiments, while Section \ref{sec:conc} concludes the paper. Supplement contains results omitted from the main body. The rest of this section introduces some notation.

\textbf{Notation.} $\N$, $\R$ and $\R^d$ denote positive integers, real numbers and $d$-dimensional vectors. For $m \in \N$, $[m]$ denotes positive integers up to and including $m$. $\langle \cdot,\cdot\rangle$ is the Euclidean inner product, while $\|\cdot\|$ is the induced vector/matrix norm. Subscripts denote agents and superscripts denote the iteration, e.g., $\xit$ refers to agent $i$'s model in iteration $t$. The ``big O'' notation $\bigO(\cdot)$ hides global constants, while $\wbigO(\cdot)$ additionally hides poly-logarithmic factors, unless stated otherwise.

\section{Preliminaries}\label{sec:prelim}

In this section we outline the preliminaries. Subsection \ref{subsec:problem} introduces the problem, while Subsection \ref{subsec:method} presents the clipped \dsgd method.

\subsection{Problem Setup}\label{subsec:problem}

Formally, the problem we aim to solve is given by
\begin{equation}\label{eq:problem}
    \argmin_{x \in \R^d}\bigg\{ f(x) = \frac{1}{n}\sum_{i \in [n]} f_i(x)\bigg\},
\end{equation} where $x \in \R^d$ represents model parameters, $n \geq 2$ is the number of agents, while $f_i: \R^d \mapsto \R$ is the cost of agent $i$. The communication network is modeled via a connected graph $G = (V,E)$, where $V = [n]$ is the set of agents, while $E \subseteq V \times V$ is the set of communication links. The graph $G$ induces a weight matrix $W \in \R^{n \times n}$, whose $(i,j)$-th entry $[W]_{ij} = w_{ij}$ is strictly positive if $\{i,j\} \in E$, otherwise $w_{ij} = 0$.

As discussed in the introduction, the goal is to solve \eqref{eq:problem} in the presence of heavy-tailed gradient noise. To that end, we assume that agents have access to a Stochastic First-order Oracle (\sfo), which, when queried by agent $i \in [n]$ with input $x \in \R^d$, returns $g_i \in \R^d$, an unbiased estimator of $\nabla f_i(x)$. The \sfo model is widely used and subsumes both batch (i.e., offline) and streaming (i.e., online) learning settings.

\subsection{Proposed Method}\label{subsec:method}

To solve \eqref{eq:problem} under heavy-tailed noise, we use the clipped \dsgd method, which consists of the following steps. At the beginning agents choose a fixed step-size $\alpha > 0$ and clipping radius $\gamma > 0$, and initial models $x^1_i \in \R^d$, $i \in [n]$. In iteration $t \geq 1$, each agent $i \in [n]$ queries the \sfo with their current model $\xit$ and receives $\git$. Agents then clip their local stochastic gradient and update the local model via the rule
\begin{equation}\label{eq:gt-model-update}
    \xitp = \sum_{j \in \calN_i}w_{ij}\big(x^{t}_j - \alpha\tgjt\big),
\end{equation} where $\calN_i \coloneqq \{j \in V: \{i,j\} \in E\} \cup \{i\}$ are the agents that can communicate with agent $i$ (itself included), while $\tgjt = \clip(g_j^t) = \min\Big\{1,\frac{\gamma}{\|g_j^t\|} \Big\}g_j^t$ is the clipped estimator. The procedure is summarized in Algorithm \ref{alg:dsgd}. The variant of clipped \dsgd studied in our work uses the adapt-then-combine (i.e., diffusion) approach, e.g., \citep{diffusion1,diffusion3} and differs from the version in \cite{clipped-dsgd1}, which uses the combine-then-adapt approach, e.g., \citep{nedic-subgrad,consensus+innovation2}.

\begin{algorithm}[tb]
\caption{Clipped \dsgd}
\label{alg:dsgd}
\begin{algorithmic}[1]
   \REQUIRE{Model initialization $x^1_i \in \R^{d}$, step-size $\alpha > 0$ and clipping radius $\gamma > 0$;}
   \FOR{$t = 1,2,\ldots$, each agent $i \in [n]$ in parallel}
        \STATE Query the \sfo with $\xit$ to obtain $\git$;  

        \STATE Perform clipping: $\tgit = \min\Big\{1,\frac{\gamma}{\|\git\|}\Big\}\git$;

        \STATE Update the model: $\xitp = \sum_{j \in \calN_i}w_{ij}\big(x_j^t - \alpha\tgjt \big)$;
    \ENDFOR
\end{algorithmic}
\end{algorithm}

\section{Main Results}\label{sec:main}

In this section we provide the main results. Subsection \ref{subsec:assump} states the assumptions, while Subsection \ref{subsec:guarantees} presents the guarantees and discusses the results.

\subsection{Assumptions}\label{subsec:assump}

We start with the assumption on the weight matrix.

\begin{assumption}\label{asmpt:network}
    The weight matrix $W \in \R^{n \times n}$ is primitive, doubly stochastic, and it holds that $\|W - \frac{1}{n}\mathbf{1}_n\mathbf{1}_n^\top\| < 1$, where $\mathbf{1}_n \in \R^n$ is the vector of all ones.
\end{assumption}

Assumption \ref{asmpt:network} covers weight matrices induced by undirected graphs and a class of strongly-connected digraphs with doubly stochastic weights, e.g., \cite{ran-vr}. Denote by $\lambda \coloneqq \|W - J\|$ , where $J \coloneqq \frac{1}{n}\mathbf{1}_n\mathbf{1}_n^\top \in \R^{n \times n}$ is the ideal communication matrix. Assumption \ref{asmpt:network} implies that $\lambda \in [0,1)$, e.g., \cite{Horn_Johnson_2012}, and the quantity $1 - \lambda$ is known as the \emph{network spectral gap}. Next, we make a technical assumption.

\begin{assumption}\label{asmpt:initialization}
    Agents' initial models $x_i^1 \in \R^d$, $i \in [n]$, are deterministic quantities. Moreover, agents share the initialization, i.e., $x_i^1 = x_j^1$, for all $i, j \in [n]$.
\end{assumption}

Assumption \ref{asmpt:initialization} allows agents to initialize their model at any deterministic and jointly selected vector. The requirement for shared initialization is made to simplify the analysis and can be realized, e.g., by running a consensus algorithm before the start of training. The next assumption states the conditions on the cost functions.

\begin{assumption}\label{asmpt:cost}
    The global cost is bounded from below and agents' costs have $L$-Lipschitz gradients, i.e., $\fs \coloneqq \inf_{x \in \R^d}f(x) > -\infty$ and for all $i \in [n]$ and $x, y\in \R^d$, we have $\|\nabla f_i(x) - \nabla f_i(y)\| \leq L\|x - y\|$.
\end{assumption}

Assumption \ref{asmpt:cost} is standard for non-convex costs, e.g., \cite{ghadimi2013stochastic}. Further, it is known that \dsgd requires a heterogeneity bound to ensure convergence under non-convex costs, e.g., \cite{NIPS2017_f7552665,pmlr-v119-koloskova20a,vlaski-dsgd-nonconv-1,vlaski-dsgd-nonconv-2}. To that end, we introduce the following condition.

\begin{assumption}\label{asmpt:bdd-het}
    There exist constants $A,B \geq 0$ (with at least one strictly positive), such that for all $x \in \R^d$ and $i \in [n]$, we have $\|\nabla f_i(x)\|^2 \leq A^2 + B^2\|\nabla f(x)\|^2$.
\end{assumption}

Assumption \ref{asmpt:bdd-het} is a standard heterogeneity condition, e.g., \citep{vlaski-dsgd-nonconv-1,vlaski-dsgd-nonconv-2,armacki2026dsgd}. It is strictly weaker than the heterogeneity bounds used in other works on decentralized clipping methods, e.g., \cite{clipped-dsgd1,li2025porter},\footnote{\label{fnote:het}\cite{clipped-dsgd1} use uniformly bounded gradients, implying Assumption \ref{asmpt:bdd-het} with $B = 0$. \cite{li2025porter} use a condition that implies Assumption \ref{asmpt:bdd-het} with $A = 0$ and $B = \frac{13}{12}$, effectively requiring that every stationary point of $f$ also be a stationary point for each $f_i$, $i \in [n]$.} however, slightly stronger than the averaged heterogeneity condition used in, e.g., \citep{pmlr-v119-koloskova20a}. For a discussion on the necessity of this condition, see Supplement \ref{sup:het}. Prior to stating the noise assumption, denote by $z_i^t \coloneqq g_i^t - \nabla f_i(x_i^t)$ the gradient noise of agent $i$ at time $t$, and let $\Ft \coloneqq \sigma\big(\{\{x_i^1\}_{i \in [n]},\ldots,\{x_i^t\}_{i \in [n]}\}\big)$ be the $\sigma$-algebra induced by agents' models up to $t$.

\begin{assumption}\label{asmpt:noise}
    At any time $t \geq 1$, the stochastic quantities satisfy the following.
    \vspace{-0.5em}
    \begin{enumerate}[leftmargin=*, label=\arabic*.]\setlength\itemsep{0.25em}
        \item The random samples $\{\xi^k_i \}_{i \in [n],k \in [t]}$ are independent across agents and iterations. 
        
        \item The noise at each agent is conditionally zero-mean, i.e., $\E[\zit \: \vert \: \Ft ] = 0$, for all $i \in [n]$.

        \item The noise at each agent has bounded $p$-th moment, $p \in (1,2]$, i.e., $\E\big[\|\zit\|^p \: \vert \: \Ft \big] \leq \sigma_i^p$, for all $i \in [n]$.
    \end{enumerate}
\end{assumption}
    
The first two conditions in Assumption \ref{asmpt:noise} are standard, while the third requires the noise to have a bounded $p$-th moment for some $p \in (1,2]$, and is widely used when studying heavy-tailed noise, e.g., \citep{zhang2020adaptive,sadiev2023highprobability,hubler2025normalization}.

\subsection{Theoretical Guarantees}\label{subsec:guarantees}

We start by defining some useful notation. Let $\onft \coloneqq \frac{1}{n}\sum_{i \in [n]}\nabla f_i(\xit) $ and $\ogt \coloneqq \frac{1}{n}\sum_{i \in [n]}\git$ denote the averaged gradient and stochastic gradient, respectively, and let $\ozt \coloneqq \ogt - \onft = \frac{1}{n}\sum_{i \in [n]}\zit$ be the averaged noise. Next, denote by $\hgt \coloneqq \frac{1}{n}\sum_{i \in [n]}\tgit$ the averaged clipped estimator, and define the averaged unbiased and biased components stemming from clipping as $\hgtu \coloneqq \hgt - \E[\hgt \: \vert \: \Ft]$ and $\hgtb \coloneqq \E[\hgt \: \vert \: \Ft] - \onft$. Similarly, let $\tgitu \coloneqq \tgit - \E[\tgit \: \vert \: \Ft]$ and $\tgitb = \E[\tgit \: \vert \: \Ft] - \nabla f_i(\xit)$ be the local unbiased and biased components. It can be seen that $\hgt - \onft = \hgtu + \hgtb$ and $\tgit - \nabla f_i(\xit) = \tgitu + \tgitb$, $i \in [n]$. We next state an important result.

\begin{lemma}\label{lm:noise-behavior}
    Let $X_1,\ldots,X_n \in \R^d$ be independent random vectors and let $\widetilde{X}_i = \clip(X_i)$, $i \in [n]$, with $\widetilde{X} = \frac{1}{n}\sum_{i \in [n]}\widetilde{X}_i$. If for some $p \in (1,2]$ and all $i \in [n]$ we have $\E\|X_i - \E[X_i]\|^p \leq \sigma_i^p$ and $\|\E[X_i]\| \leq \nicefrac{\gamma}{2}$, then
    \begin{align*}
        \E\|\widetilde{X} - \E[\widetilde{X}]\|^2 \leq \frac{18\gamma^{2-p}}{n^2}\sum_{i \in [n]}\sigma_i^p.
    \end{align*}
\end{lemma}

Denoting by $\sigma^p \coloneqq \frac{1}{n}\sum_{i \in [n]}\sigma_i^p$, Lemma \ref{lm:noise-behavior} tells us that the variance of the averaged clipped estimator scales as $\bigO\big(\frac{\sigma^p}{n}\big)$, and is crucial in establishing linear speed-up. A similar result appeared in the client-server setup in \cite[Lemma B.3]{pmlr-v235-gorbunov24a}, provided for illustrative purposes while using a different approach to show linear speed-up. On the other hand, Lemma \ref{lm:noise-behavior} is a key result for achieving linear speed-up in our work. In Supplement \ref{sup:proof} we provide an extended version of Lemma \ref{lm:noise-behavior}. Next, define the averaged model $\oxt \coloneqq \frac{1}{n}\sum_{i \in [n]}\xit$. We then have the following descent inequality.

\begin{lemma}\label{lm:descent-inequality}
    If Assumption \ref{asmpt:cost} holds and $\alpha \leq \frac{1}{2L}$, then for any $t \geq 1$
    \begin{align*}
        f(&\oxtp) \leq f(\oxt) - \frac{\alpha}{4}\|\nabla f(\oxt)\|^2 - \alpha\langle \nabla f(\oxt), \hgtu\rangle \\ 
        &+ 2\alpha^2L\|\hgtu\|^2 + 2\alpha\|\hgtb\|^2 + \frac{\alpha L^2}{2n}\sum_{i \in [n]}\|\xit - \oxt\|^2. 
    \end{align*}
\end{lemma}

Lemma \ref{lm:descent-inequality} is the starting point in our analysis. In addition to the standard terms appearing in centralized clipping analysis, e.g., \citep{sadiev2023highprobability,nguyen2023improved}, the inequality contains the consensus gap term, which we aim to bound next. We start with a simple but useful bound on the consensus gap.

\begin{lemma}\label{lm:consensus-crude}
    If Assumptions \ref{asmpt:network} and \ref{asmpt:initialization} hold, then for all $i \in [n]$ and any $t \geq 1$ 
    \begin{align*}
        &\|\xitp - \oxtp\| \leq \frac{\alpha\gamma\lambda\sqrt{n}}{1-\lambda}.
    \end{align*}
\end{lemma}

Lemma \ref{lm:consensus-crude} provides an initial deterministic bound on the consensus gap and is useful since it indicates that, for a sufficiently small step-size, the consensus gap can be controlled by the clipping radius. However, only using Lemma \ref{lm:consensus-crude} would propagate the network effect to the leading term and result in the loss of linear speed-up in the final bound (see Remark \ref{rmk:consensus} in Supplement \ref{sup:proof} for details). To fix this issue, we next provide an additional refined bound on the consensus gap.

\begin{lemma}\label{lm:consensus-bound}
    If Assumptions \ref{asmpt:network}-\ref{asmpt:bdd-het} hold, then for all $t \geq 1$
    \begin{align*}
        &\frac{1}{n}\sum_{i \in [n]}\|\xitp - \oxtp\|^2 \leq \frac{8\alpha^2\lambda^2A^2}{(1-\lambda^2)}\sum_{k = 1}^t\bigg(\frac{1+\lambda^2}{2}\bigg)^{t-k} \\
        &+ \frac{8\alpha^2\lambda^2B^2}{n(1-\lambda^2)}\sum_{k = 1}^t\bigg(\frac{1+\lambda^2}{2}\bigg)^{t-k}\sum_{i \in [n]}\|\nabla f(\xik)\|^2 \\
        &+ \frac{8\alpha^2\lambda^2}{n(1-\lambda^2)}\sum_{k = 1}^t\bigg(\frac{1+\lambda^2}{2}\bigg)^{t-k}\sum_{i \in [n]}\big(\|\tgikb\|^2 + \|\tgiku\|^2\big).
    \end{align*}
\end{lemma}

Lemma \ref{lm:consensus-bound} establishes a more fine-grained result on the average consensus gap, capturing the effects of heterogeneity, network connectivity and local clipping. Lemmas \ref{lm:noise-behavior}-\ref{lm:consensus-bound} together are crucial in achieving our main results, stated next. We start with HP convergence. 

\begin{theorem}\label{thm:non-cvx-hp}
    Let Assumptions \ref{asmpt:network}-\ref{asmpt:noise} hold. If for any $T \geq 1$ and $\delta \in (0,1)$, the clipping radius and the step-size are chosen as $\gamma = \max\Big\{2A + 4B\sqrt{L\Delta_1} + \frac{2\lambda L\sqrt{n}}{1-\lambda}, (nT)^{\frac{1}{3p-2}} \Big\}$ and $\alpha = \min\Big\{\frac{1}{2L}, \frac{1-\lambda^2}{8\lambda LB\sqrt{3}}, \frac{1-\lambda^2}{\lambda L\sqrt{6}}, \frac{(1-\lambda^2)^2}{6\lambda^2L}, \frac{(1-\lambda^2)^{2/3}\Delta_1^{1/3}}{2(\lambda LA)^{2/3}T^{1/3}\sqrt[3]{12}}, \frac{C_1}{\gamma}, \frac{C_2}{\gamma}n^{\frac{2p-1}{3p-2}}T^{\frac{1-p}{3p-2}}, \frac{C_3}{\gamma^{2/3}}n^{\frac{p}{3(3p-2)}}T^{\frac{2(1-p)}{3(3p-2)}} \Big\}$, with $C_1 = \min\Big\{1,\frac{3\sqrt{\Delta_1/2L}}{64\log\big(\frac{6T}{\delta}\big)}\Big\}$, $C_2 = \min\Big\{\frac{\Delta_1}{64\sigma^{2p}}, \frac{\sqrt{\Delta_1/2L}}{192\log\big(\frac{6T}{\delta}\big)\sigma^{p/2}} \Big\}$ and $C_3 = \frac{\Delta_1^{1/3}(1-\lambda^2)^{2/3}}{12\sigma^{p/3}\lambda^{2/3}L^{2/3}}$, then with probability at least $1 - \delta$
    \begin{align*}
        \frac{1}{nT}\sum_{t \in [T]}\sum_{i \in [n]}\|\nabla f(\xit)\|^2 = \wbigO\bigg((nT)^{\frac{2(1-p)}{3p-2}}\bigg),
    \end{align*} where $\wbigO(\cdot)$ hides terms of higher-order in $T$.
\end{theorem}

The next result establishes the IE convergence bound.

\begin{theorem}\label{thm:non-cvx-ie}
    Let Assumptions \ref{asmpt:network}-\ref{asmpt:noise} hold. If for any $T \geq 2$ the clipping radius and the step-size are chosen as $\gamma = \max\Big\{2A + 4B\sqrt{L\Delta_1} + \frac{2\lambda L\sqrt{n}}{1-\lambda}, (nT)^{\frac{1}{3p-2}} \Big\}$ and $\alpha = \min\Big\{\frac{1}{2L}, \frac{1-\lambda^2}{8\lambda LB\sqrt{3}}, \frac{1-\lambda^2}{\lambda L\sqrt{6}}, \linebreak \frac{(1-\lambda^2)^2}{6\lambda^2L}, \frac{(1-\lambda^2)^{2/3}\Delta_1^{1/3}}{2(\lambda LA)^{2/3}T^{1/3}\sqrt[3]{12}}, \frac{C_1}{\gamma}, \frac{C_2}{\gamma}n^{\frac{2p-1}{3p-2}}T^{\frac{1-p}{3p-2}}, \frac{\Delta_1^{1/3}(1-\lambda^2)^{2/3}}{12\sigma^{p/3}\gamma^{2/3}\lambda^{2/3}L^{2/3}}n^{\frac{p}{3(3p-2)}}T^{\frac{2(1-p)}{3(3p-2)}}\Big\}$, where $C_1 = \linebreak \min\Big\{1,\frac{3\sqrt{\Delta_1/2L}}{64\log(6n^4T^5)}\Big\}$ and $C_2 = \min\Big\{\frac{\Delta_1}{64\sigma^{2p}},\frac{\sqrt{\Delta_1/2L}}{192\log(6n^4T^5)\sigma^{p/2}} \Big\}$, it then holds that
    \begin{align*}
        \frac{1}{nT}\sum_{t \in [T]}\sum_{i \in [n]}\E\|\nabla f(\xit)\|^2 = \wbigO\bigg((nT)^{\frac{2(1-p)}{3p-2}}\bigg),
    \end{align*} where $\wbigO(\cdot)$ hides terms of higher-order in $T$.
\end{theorem}

We next discuss several aspects of our results.

\textbf{Convergence Rates}. Theorems \ref{thm:non-cvx-hp} and \ref{thm:non-cvx-ie} respectively establish HP and IE convergence of clipped \dsgd under heavy-tailed noise. In both cases, the leading term with respect to $T$ is of the order $\wbigO\Big((nT)^{\frac{2(1-p)}{3p-2}}\Big)$, which has three important implications. First, the rate in $T$ matches the \emph{optimal rate} achieved in centralized settings \citep{zhang2020adaptive,nguyen2023improved}. Second, \emph{linear speed-up in the number of agents} $n$ is achieved, exhibiting the same exponent with respect to $p$. Third, the leading term \emph{has no network dependence}, with the spectral gap only affecting higher-order terms. For full rates, which contain several higher-order terms, the reader is referred to Supplement \ref{sup:proof}.

\textbf{Clipping Radius}. The clipping radius in our work is of the form $\gamma = \max\big\{C_d, (nT)^{\frac{1}{3p-2}} \big\}$, where $C_d >0$ is a problem-related constant. Compared to centralized works \citep{zhang2020adaptive,nguyen2023improved}, where the clipping radius is also of the form $\gamma = \max\big\{C_c, T^{\frac{1}{3p-2}}\big\}$, we note a few things. First, both clipping radii have the same structure, i.e., a maximum of a problem-related constant and a term growing with $T$. Second, our clipping radius increases with higher heterogeneity (via $A,B$) and worse network connectivity (via the spectral gap), which is expected, as these quantities reflect the difficulty of the decentralized problem. Third, our radius also increases with $n$, which is again expected in the multi-agent setting, and is consistent with \cite{pmlr-v235-gorbunov24a}, who similarly use a clipping radius that increases with $n$. 

\textbf{Step-size.} The step-size is of the form $\alpha = \min\big\{S_d, \frac{C_1}{\gamma},\frac{C_2}{\gamma}n^{\frac{2p-1}{3p-2}}T^{\frac{1-p}{3p-2}},\frac{C_3}{\gamma^{2/3}}n^{\frac{p}{3(3p-2)}}T^{\frac{2(1-p)}{3(3p-2)}} \big\}$, where $S_d > 0$ is a problem-related constant, corresponding to standard conditions on the step-size for vanilla \dsgd, e.g., \citep{armacki2026dsgd}, while the remaining terms stem from the use of clipping. The first two, scaling as $\propto \gamma^{-1}$, are consistent with the step-size choice in \cite{zhang2020adaptive,nguyen2023improved} (up to $n$ factors), while the third term, scaling as $\propto \gamma^{-2/3}$, is novel and specific to our decentralized settings, used to optimize the network dependence and transient time.

\textbf{Transient Time.} Using our rates, one can derive the \emph{transient time} for clipped \dsgd, i.e., the time needed to achieve a global, network-independent $\bigO\Big((nT)^{\frac{2(1-p)}{3p-2}}\Big)$ rate. Focusing on the effects of the number of agents and the spectral gap, our transient time is of the order $\wbigO\Big(\max\Big\{n^{\frac{2p-1}{p-1}},\frac{n^{\frac{3p-4}{2}}}{(1-\lambda)^{3p-2}}, \frac{n^{\frac{8(p-1)}{2+p}}}{(1-\lambda)^{\frac{4(3p-2)}{2+p}}}, \frac{n^{\frac{5p-4}{p}}}{(1-\lambda)^{\frac{2(3p-2)}{p}}}, \frac{n^{3(p-1)}}{(1-\lambda)^{3p-2}}, \frac{n^{\frac{2(p-1)}{p}}}{(1-\lambda)^{\frac{2(3p-2)}{p}}}, \frac{n^{\frac{7p-6}{2p}}}{(1-\lambda)^{\frac{3p-2}{p}}} \Big\}\Big)$. For $p = 2$ (i.e., bounded variance), it evaluates to $\bigO\Big(\frac{n^3}{(1-\lambda)^4}\Big)$, matching the transient time of vanilla \dsgd under bounded variance/sub-Gaussian noise, e.g., \citep{pmlr-v119-koloskova20a,unified-refined,armacki2026dsgd}. For further discussion, see Supplement \ref{sup:tran-time}. 

\begin{figure*}[t]
\centering
\captionsetup{font=footnotesize}
\begin{tabular}{cc}
\includegraphics[width=0.4\linewidth]{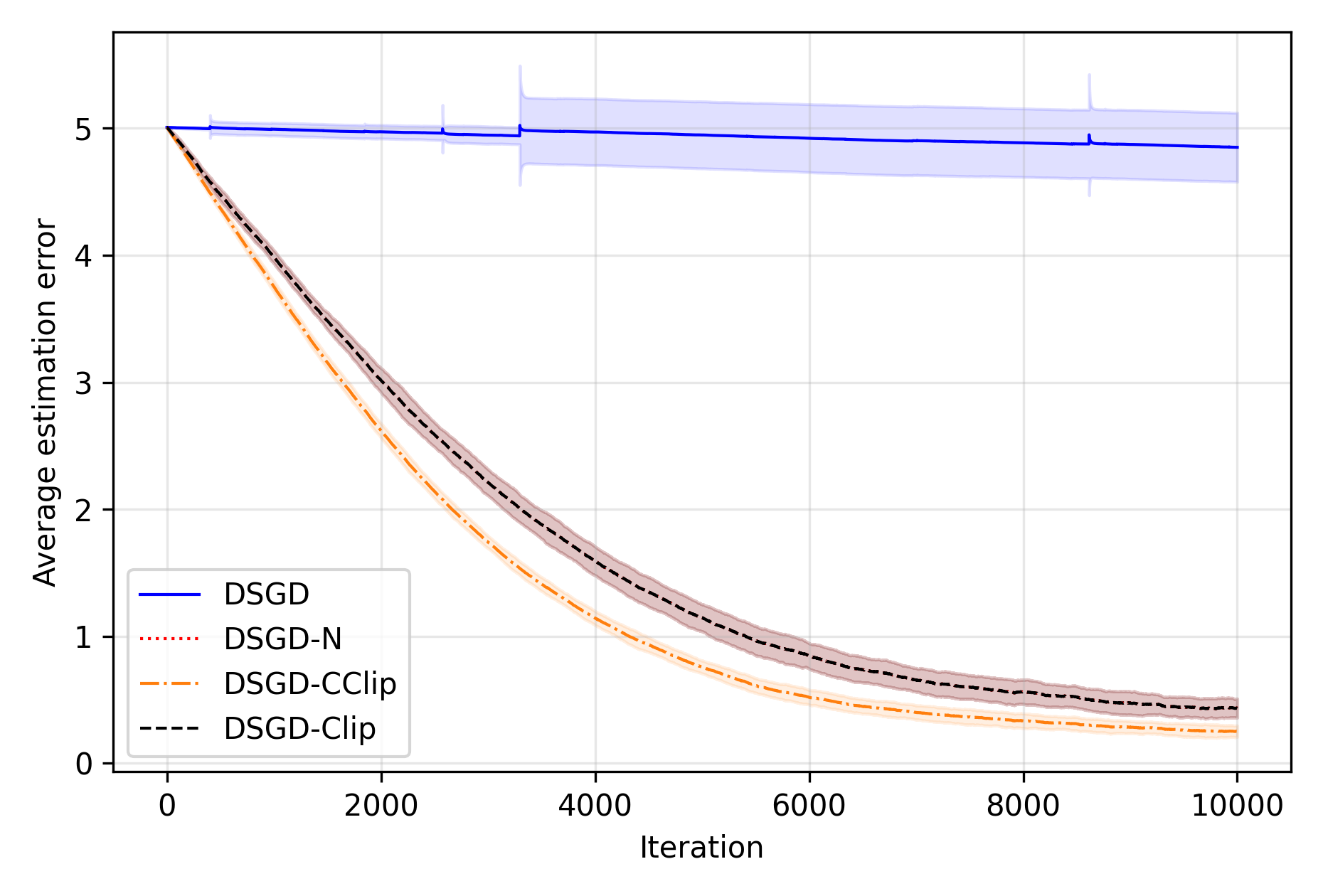}
&
\includegraphics[width=0.4\linewidth]{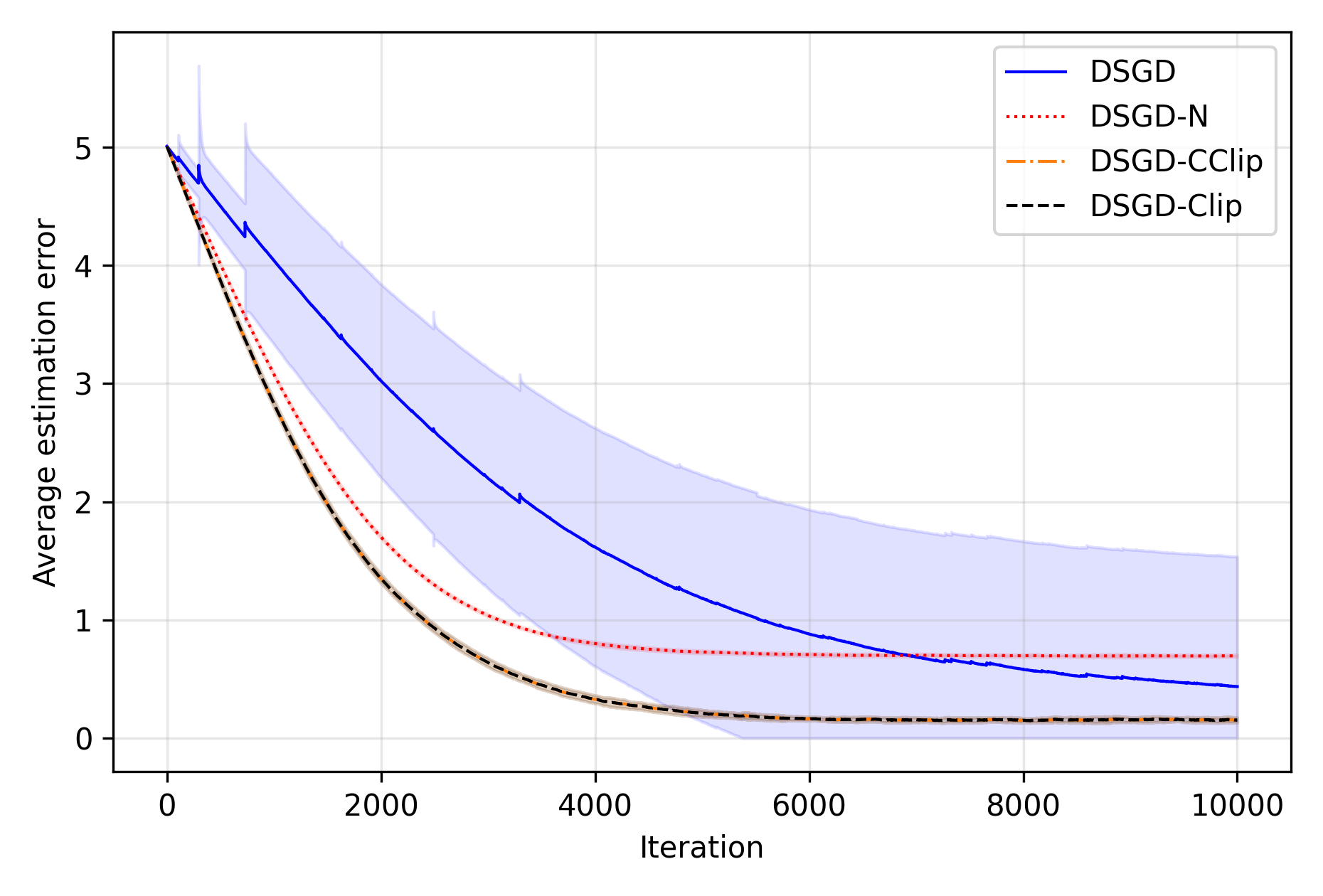}
\end{tabular}
\caption{
Convergence of decentralized methods under heavy-tailed gradient noise. Left: Student's
$t$ noise; right: L\'evy $\alpha$-stable noise. Solid curves show the mean estimation error, while shaded regions indicate one standard deviation.}
\label{fig:conv}
\end{figure*}

\textbf{Comparison with Existing Works.} Compared to \cite{clipped-dsgd1}, who study a variant of clipped \dsgd, we relax the bounded gradients condition and achieve improved oracle complexity. Compared to \cite{li2025porter}, who study clipped GT with compression, EF and mini-batches under bounded variance noise, we consider a much simpler algorithm under more general noise and heterogeneity conditions, and achieve improved oracle complexity. Compared to \cite{yu2026decentralized}, who study normalized GT with local momentum, we use the baseline clipped \dsgd, achieving better oracle complexity. Finally, compared to \cite{luoluo2026near}, who study normalized GT with large mini-batches and multiple communication rounds per iteration, we achieve the same optimal oracle complexity while using a significantly simpler algorithm, with only one stochastic gradient evaluation and one communication round per iteration. We note that the GT mechanism allows \cite{yu2026decentralized,luoluo2026near} to remove the heterogeneity bound, which is not the case in \citep{li2025porter} despite the use of GT, or in our work and \citep{clipped-dsgd1}, where \dsgd necessitates a heterogeneity bound to ensure convergence. Compared to all of the said works, which provide \emph{either HP or IE} results, we show \emph{both HP and IE} results.

\textbf{Insights.} As discussed in the introduction and shown in Figure \ref{fig:normalization_clipping_example}, normalized \dsgd can fail to converge even with full gradients on a simple quadratic cost. On the other hand, our work demonstrates that clipped \dsgd achieves optimal convergence rates with linear speed-up under heavy-tailed noise, highlighting a clear gap between these two nonlinearities in decentralized settings. Intuitively, the gap stems from the fact that, while normalization always modifies the gradient direction, clipping keeps it unchanged when the gradient norm is below a threshold, maintaining full magnitude information. As such, normalization always biases the network-averaged gradient, while clipping can preserve the true network-averaged gradient, allowing clipped \dsgd to achieve order-optimal complexity with linear speed-up, without needing mechanisms like GT, accelerated gossip and local momentum/mini-batches.

\section{Numerical Results}\label{sec:num}

\begin{figure*}[t]
\centering
\captionsetup{font=footnotesize}
\begin{tabular}{cc}
\includegraphics[width=0.4\linewidth]{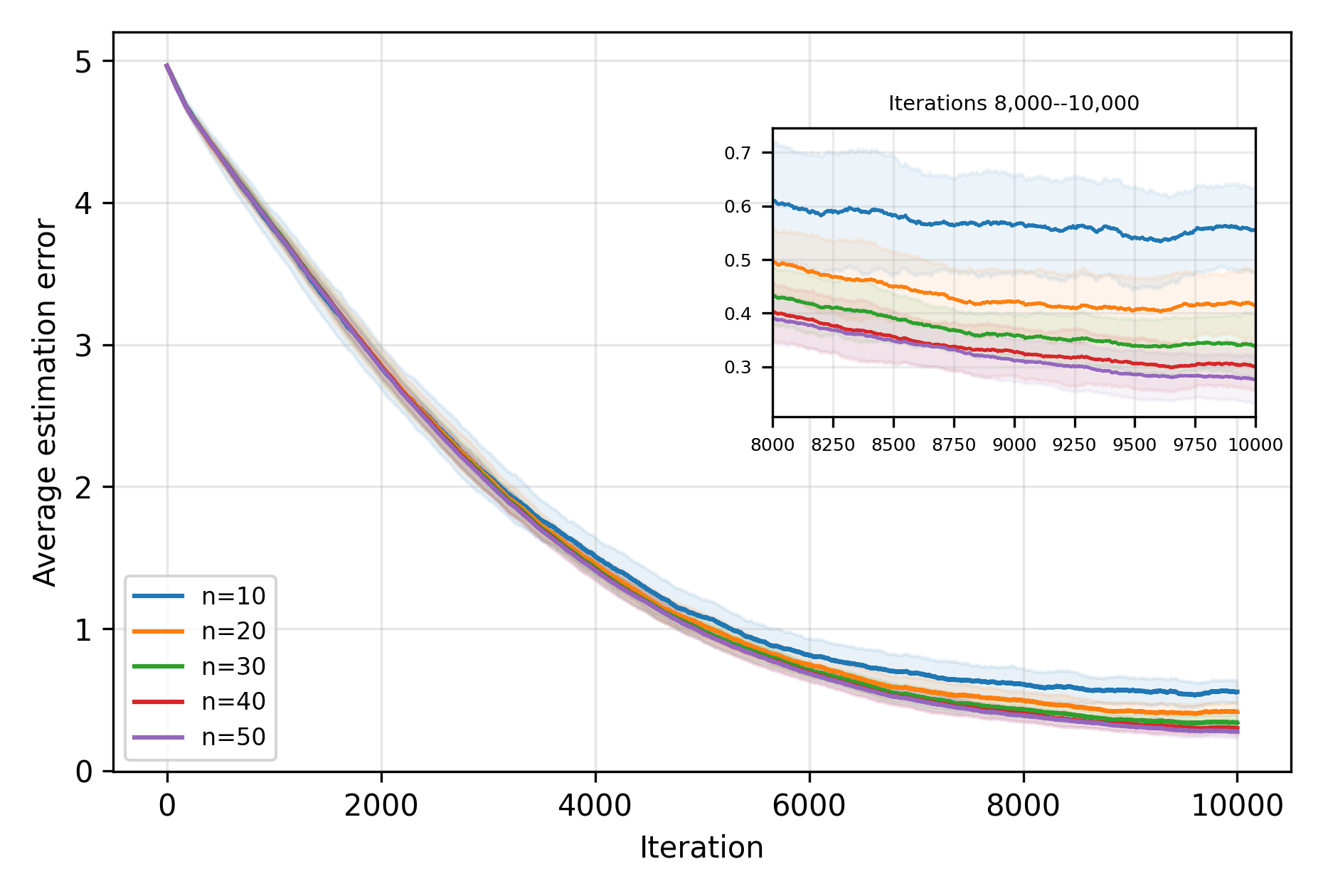}
&
\includegraphics[width=0.4\linewidth]{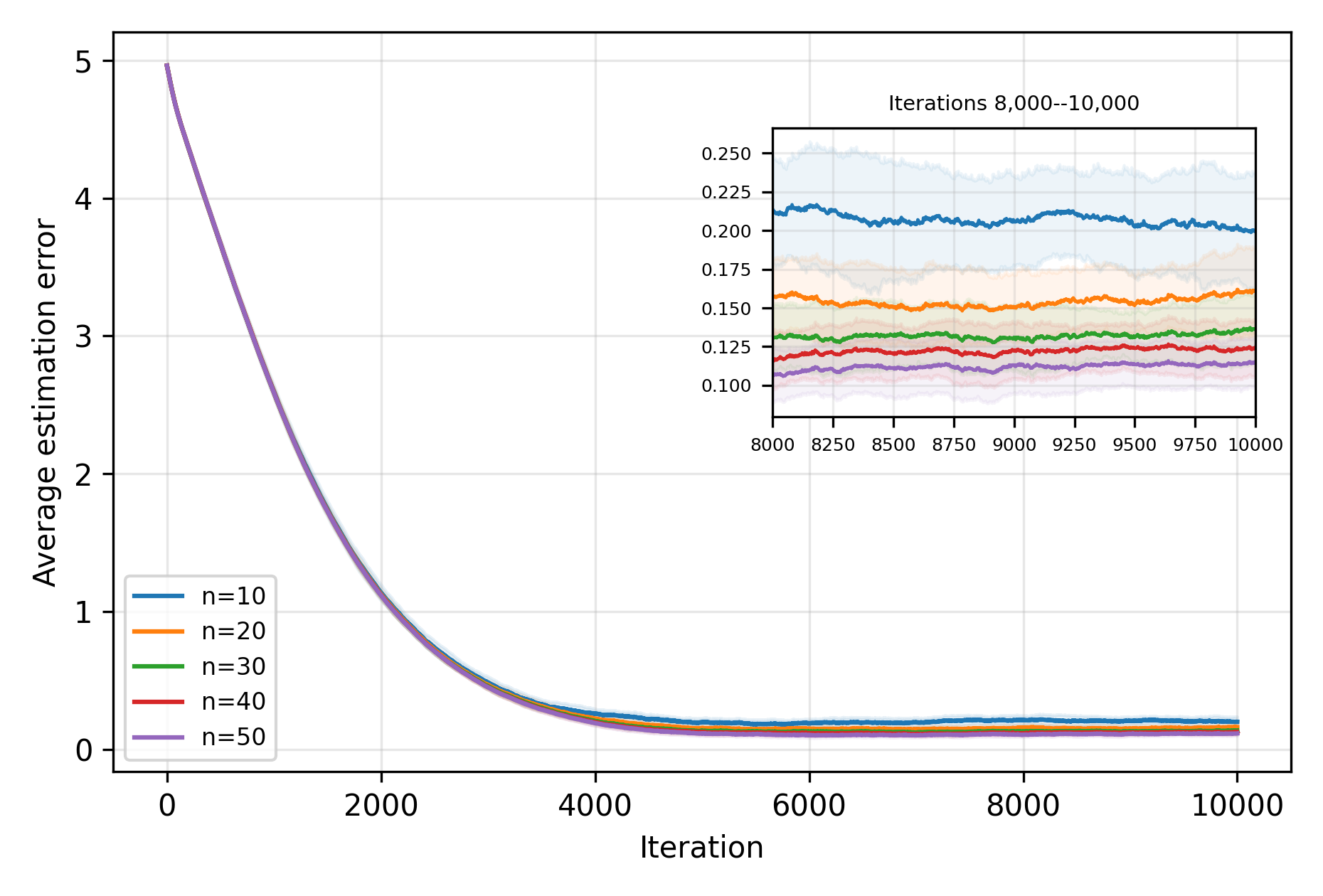}
\end{tabular}
\caption{Convergence of \dsgdclip under heavy-tailed noise for networks of size $n \in \{10,20,30,40,50\}$. Left: Student's $t$ noise; right: L\'evy $\alpha$-stable noise. We can see that \dsgdclip accelerates with network size, achieving linear speed-up.}
\label{fig:lin_speedup}
\end{figure*}

In this section we present numerical experiments. We refer the reader to Supplement \ref{sup:num} for full details on experimental setup and additional results.   

\textbf{Linear Regression on Synthetic Data.} We consider the problem of robust linear regression to evaluate two main aspects of our results: convergence and linear speed-up under heavy-tailed noise. We compare the performance of vanilla \dsgd to nonlinear variants using: clipping (\dsgdclip), component-wise clipping (\dsgdcclip) and normalization (\dsgdn), with hyperparameters selected through extensive grid search. To evaluate convergence under heavy-tailed noise, we consider a network of $n = 20$ agents communicating over a ring graph. Each agent is assigned a dataset of $50$ feature-response pairs $(A,b) \in \R^{20} \times \R$, where each feature component of $A$ is sampled independently from a $\bern(q)$ distribution, where $q = 0.9$ for the first two, $q=0.5$ for the next two, and $q = 0.1$ for the remaining components, the ground truth model $x^\star$ is drawn from $\calN(0,I_{20})$, while the responses satisfy $b=\langle A,x^\star\rangle$. Agents' local costs correspond to Tukey's biweight loss and local gradients are perturbed by Student's $t$ (degrees of freedom $1.5$, scale $1.0$) and
L\'evy $\alpha$-stable noise (stability $1.5$, skewness $0.5$, scale $1.0$), with local models initialized at the zero vector. We measure performance via the average estimation error $\mathcal{E}_t=\frac{1}{n}\sum_{i=1}^{n}\left\|w_i^t-w^\star\right\|_2$, evaluated over $30$ independent
runs. The results are presented in Figure \ref{fig:conv}. We can observe that nonlinear methods stabilize and accelerate the convergence of \dsgd, highlighting the importance of nonlinear methods under heavy-tailed noise, and that both \dsgdclip and \dsgdcclip perform on par with or better than \dsgdn.  

To evaluate the linear speed-up of \dsgdclip under heavy tailed noise, we consider five networks of $n \in \{10,20,30,40,50\}$ agents communicating over an undirected Erd\H{o}s--R\'enyi graph. To better isolate the linear speed-up effect, we set the edge probability of the graph at $0.6$ while ensuring that $\lambda = 0.8$ across every network size, thus fixing the network spectral gap. We generate a total of $1200$ feature-response pairs using the same rules as in the previous experiments and spread them evenly among agents. We again inject Student's $t$ and L\'{e}vy $\alpha$-stable noise to agents' local gradients, using the same Tukey loss, model initialization and performance measure. The results are presented in Figure \ref{fig:lin_speedup}. We can see that \dsgdclip consistently converges faster as the network size increases, confirming that linear speed-up is achieved.

\textbf{Transformer Training on Real Data.} Finally, we test the performance of nonlinear decentralized methods on the problem of decentralized training of a decoder-only transformer with real data. The simulations are conducted using antup and additional some additional results A100 NVIDIA GPU. We consider an undirected ring network of $n = 8$ agents, trained on the ``Multi30k'' and ``Tiny Shakespeare'' datasets. ``Multi30k'' contains 29,000 training sentences, yielding approximately 378,000 word tokens after tokenization, while ``Tiny Shakespeare'' contains about 1 million training characters. We use two-layer decoder-only transformers with hidden dimension 128, four attention heads, feedforward dimension 512, and dropout 0.1. The models are trained using a batch size of 64 sequences per agent, with context length 64 and stride 32. The training data are partitioned into equally sized contiguous shards, and all agents are initialized with the same random model. We again consider \dsgd, \dsgdn, \dsgdclip, and \dsgdcclip methods and measure the performance via the log-perplexity achieved on a validation set. Table \ref{tab:transformers} reports the final validation loss averaged over five independent runs. We can see that \dsgdcclip achieves the lowest validation loss on both datasets, while \dsgdclip and \dsgdn perform on par with one another. Importantly, all three methods improve over vanilla \dsgd.

\begin{table}[htbp]
\centering
\captionsetup{font=normalsize}
\caption{Validation loss after 20 epochs on ``Multi30k'' and ``Tiny Shakespeare'' datasets. We report the mean and standard deviation over five independent runs.}
\label{tab:transformers}
\begin{tabular}{|l|c|c|}
\hline
Algorithms & ``Multi30k'' & ``Tiny Shakespeare'' \\
\hline
\dsgd & $4.634{\pm0.026}$ & $2.357{\pm0.006}$ \\
\hline
\dsgdn & $4.386{\pm0.016}$ & $1.966{\pm0.035}$ \\
\hline
\dsgdclip & $4.375{\pm0.021}$ & $1.980{\pm0.050}$ \\
\hline
\dsgdcclip & \textbf{3.962$\pm$0.008} & \textbf{1.725$\pm$0.010} \\
\hline
\end{tabular}
\end{table}

\section{Conclusion}\label{sec:conc}

In this work, we studied whether a baseline decentralized algorithm with a local nonlinear gradient transformation can attain order-optimal rates under heavy-tailed noise, without further mechanisms like mini-batches or accelerated gossip. We answered affirmatively with clipped \dsgd, establishing order-optimal convergence rates and linear speed-up for smooth non-convex costs. Our results also highlighted an important distinction between clipping and normalization in decentralized settings, raising broader questions about the effects of nonlinearities in decentralized optimization. Are some nonlinearities better suited to decentralized settings, and if so, can we establish a hierarchy across different regimes? Another important direction is whether baseline GT with clipping can simultaneously eliminate the bounded heterogeneity condition and retain order-optimal convergence rates. This, in turn, raises a further design question: should clipping be applied to local gradients or to the tracker variable, and what are the consequences of each choice? 

\subsubsection*{Acknowledgements}
The authors would like to thank Shuhua Yu (Meta FAIR) for helpful discussions.

\bibliographystyle{plainnat}
\bibliography{bibliography}

\appendix

\section*{Supplement}

The supplement contains results omitted from the main body. Section \ref{sup:lit} reviews some further literature, Section \ref{sup:proof} provides technical results and proofs, Section \ref{sup:num} details the experimental setup and contains additional results, Section \ref{sup:het} discusses the heterogeneity condition, while Section \ref{sup:tran-time} discusses the transient time.

\section{Additional Literature}\label{sup:lit}

In this section we provide some additional related literature. 

\textbf{Centralized Results.} \cite{nemirovski2009robust} establish optimal HP convergence rates of SGD for convex costs under light-tailed stochastic gradients. \cite{ghadimi2013stochastic} establish, among other results, the optimal IE and HP convergence rate of SGD for non-convex costs under light-tailed noise, while \cite{li2020high} show the same for momentum SGD. \cite{harvey2019tight} and \cite{pmlr-v206-bajovic23a} respectively provide optimal HP convergence rates for the last iterate of SGD for non-smooth and smooth strongly convex costs, while \cite{liu2024revisiting} establish unified IE and HP convergence guarantees for smooth and non-smooth convex and strongly convex costs. \cite{liu2023high} generalize the previous works on non-convex and convex costs, providing unified HP guarantees for several algorithms, including SGD and AdaGrad for smooth and non-smooth costs. \cite{madden2020high} establish HP convergence of SGD under sub-Weibull noise. \cite{gorbunov2020stochastic} propose an accelerated method with gradient clipping for smooth convex costs, while \cite{parletta2022high} study HP convergence of sub-gradient methods for non-smooth convex costs. \cite{koloskova2023grad_clip} show that clipped SGD fails to converge for any fixed constant value of the clipping threshold $\gamma$. \cite{jakovetic2023nonlinear} study convergence of a unified nonlinear SGD framework under symmetric heavy-tailed noise and strongly convex costs. Finally, \cite{gorbunov2023breaking} show that it is possible to achieve the optimal HP convergence rate for convex costs using the median-of-means gradient estimator.

\textbf{Decentralized Results.} \cite{conv-rates-dsgd-jakovetic} show that \dsgd converges at an optimal IE rate for strongly convex costs, while \cite{Pu2021} show that \dsgd with a fixed step-size and the gradient tracking (GT) mechanism, e.g., \cite{scutari-next,nedic-tracking,qu-harnessing}, converges IE to a neighborhood of the optimal solution. \cite{vlaski-dsgd-nonconv-1,vlaski-dsgd-nonconv-2} study IE guarantees of \dsgd for non-convex costs and show that it escapes saddle points with high probability. \cite{pmlr-v119-koloskova20a} provide unified IE guarantees for \dsgd with local updates and changing network topology, with optimal rates and linear speed-up in the number of users for non-convex and (strongly) convex costs. \cite{ran-improved} establish optimal IE rates of \dsgd with GT and linear speed-up in the number of users for non-convex costs and costs satisfying the Polyak-\L{}ojasiewicz (PL) condition. \cite{clipped-dual-avg,clipped-dsgd2} study HP convergence of decentralized algorithms with clipping, for convex costs under heavy-tailed noise. \cite{yu2026smoothed} study convergence of \dsgd using smooth clipping under symmetric heavy-tailed noise and strongly convex costs. Finally, a recent line of works show HP convergence with linear speed-up of vanilla \dsgd and GT \citep{armacki2026dsgd,armacki2026high}.

\section{Missing Proofs}\label{sup:proof}

In this section we provide several technical results and proofs omitted from the main body. Subsection \ref{subsup:gen} provides general results, Subsection \ref{subsup:hp} provides proofs for HP results, while Subsection \ref{subsup:ie} provides proofs for IE results. 

\subsection{General Results}\label{subsup:gen}

In this subsection we prove Lemmas \ref{lm:noise-behavior}-\ref{lm:consensus-bound}. We start by stating a useful lemma from \cite{sadiev2023highprobability}.

\begin{lemma}[Lemma 5.1 in \cite{sadiev2023highprobability}]\label{lm:sadiev}
    Let $X \in \R^d$ be a random vector and let $x = \E[X]$ and $\widetilde{X} = \clip(X)$ be its expectation and clipped version. If for some $\sigma > 0$ and $p \in (1,2]$ we have $\E\|X - x\|^p \leq \sigma^p$ and $\|x\| \leq \nicefrac{\gamma}{2}$, then
    \begin{align}
        &\|\E[\widetilde{X}] - x\| \leq \frac{4\sigma^p}{\gamma^{p-1}}, \label{eq:bias-bdd} \\
        &\E\|\widetilde{X} - \E[\widetilde{X}]\|^2 \leq 18\gamma^{2-p}\sigma^p. \label{eq:variance-bdd}
    \end{align}
\end{lemma}

We next state and prove an extended version of Lemma \ref{lm:noise-behavior}.
{
\renewcommand{\thelemma}{1}
\begin{lemma}
    Let $X_1,\ldots,X_n \in \R^d$ be independent random vectors and let $x_i = \E[X_i]$ and $\widetilde{X}_i = \clip(X_i)$, $i \in [n]$ be the expectation and clipped version, with $\overline{x} = \frac{1}{n}\sum_{i \in [n]}x_i$ and $\widetilde{X} = \frac{1}{n}\sum_{i \in [n]}\widetilde{X}_i$ being the averaged values. If for some $\sigma_i > 0$ and $p \in (1,2]$ we have $\E\|X_i - x_i\|^p \leq \sigma_i^p$ and $\|x_i\| \leq \nicefrac{\gamma}{2}$, then
    \begin{align*}
        &\|\E[\widetilde{X}] - \overline{x}\| \leq \frac{4\sigma^p}{\gamma^{p-1}}, \\
        &\E\|\widetilde{X} - \E[\widetilde{X}]\|^2 \leq \frac{18\gamma^{2-p}\sigma^p}{n}, 
    \end{align*} where $\sigma^p = \frac{1}{n}\sum_{i \in [n]}\sigma_i^p$.
\end{lemma}
}
\begin{proof}
    Using Jensen's inequality and \eqref{eq:bias-bdd}, we then get
    \begin{equation*}
        \|\E[\widetilde{X}] - \overline{x}\| \leq \frac{1}{n}\sum_{i \in [n]}\|\E[\widetilde{X}_i] - \overline{x}\| \leq \frac{1}{n}\sum_{i \in [n]}\frac{4\sigma_i^p}{\gamma^{p-1}} = \frac{4\sigma^p}{\gamma^{p-1}},
    \end{equation*} proving the first claim. Next, using the fact that $X_i$'s are independent, we have
    \begin{align*}
        \E\|\widetilde{X} - \E[\widetilde{X}]\|^2 &= \frac{1}{n^2}\sum_{i \in [n]}\E\|\widetilde{X}_i - \E[\widetilde{X}_i]\|^2 + \frac{1}{n^2}\sum_{j \neq i}\E\big\langle \widetilde{X}_i - \E[\widetilde{X}_i], \widetilde{X}_j - \E[\widetilde{X}_j]\big\rangle \\ 
        &= \frac{1}{n^2}\sum_{i \in [n]}\E\|\widetilde{X}_i - \E[\widetilde{X}_i]\|^2 \leq \frac{1}{n^2}\sum_{i \in [n]}18\gamma^{2-p}\sigma_i^p = \frac{18\gamma^{2-p}\sigma^p}{n},
    \end{align*} where the inequality follows from \eqref{eq:variance-bdd}, completing the proof.
\end{proof}

We continue by proving Lemma \ref{lm:descent-inequality}.

\begin{proof}[Proof of Lemma \ref{lm:descent-inequality}]
    Starting from the standard descent lemma for smooth costs, we get
    \begin{align*}
        f(\oxtp) &\leq f(\oxt) - \alpha \langle \nabla f(\oxt), \hgt \rangle + \frac{\alpha^2L}{2}\|\hgt\|^2 \\
        &= f(\oxt) - \alpha \langle \nabla f(\oxt), \onft \rangle - \alpha \langle \nabla f(\oxt), \hgtu + \hgtb \rangle + \frac{\alpha^2L}{2}\|\hgt\|^2 \\
        &\stackrel{(i)}{\leq} f(\oxt) - \alpha \langle \nabla f(\oxt), \onft \rangle - \alpha \langle \nabla f(\oxt), \hgtu + \hgtb \rangle + \alpha^2L\big(\|\hgtu + \hgtb\|^2 + \|\onft\|^2), \\
        &\stackrel{(ii)}{\leq} f(\oxt) - \frac{\alpha}{4}\|\nabla f(\oxt)\|^2 - \frac{\alpha}{2}\big(1-2\alpha L\big)\|\onft\|^2 +\frac{\alpha}{2}\|\nabla f(\oxt) - \onft\|^2 \\
        &- \alpha\langle \nabla f(\oxt),\hgtu\rangle + \alpha(1 + 2\alpha L)\|\hgtb \|^2 + 2\alpha^2L\|\hgtu\|^2 
    \end{align*} where in $(i)$ we used $\|\hgt\|^2 \leq 2\|\hgtu + \hgtb\|^2 + 2\|\onft\|^2$, while in $(ii)$ we applied the identities $\langle a,b\rangle = \frac{1}{2}\big(\|a\|^2 + \|b\|^2 - \|a-b\|^2)$ and $\langle a,b\rangle \leq \frac{\|a\|^2}{4} + \|b\|^2$, respectively to $\langle \nabla f(\oxt),\onft\rangle$ and $\langle \nabla f(\oxt), \hgtu + \hgtb\rangle$. Choosing $\alpha \leq \frac{1}{2L}$ implies $1 - 2\alpha L \geq 0$, which, combined with smoothness of $f$, gives
    \begin{align*}
        f(\oxtp) \leq f(\oxt) - \frac{\alpha}{4}\|\nabla f(\oxt)\|^2 - \alpha\langle \nabla f(\oxt), \hgtu\rangle + 2\alpha^2L\|\hgtu\|^2 + 2\alpha\|\hgtb\|^2 + \frac{\alpha L^2}{2n}\sum_{i \in [n]}\|\xit - \oxt\|^2.
    \end{align*} 
\end{proof}

For ease of notation, define the vectorized, stacked notation $\bxt \coloneqq \text{col}(x_1^t,\ldots,x_n^t) \in \R^{nd}$, $\wbgt \coloneqq \text{col}(\widetilde{g}_1^t,\ldots,\widetilde{g}_n^t) \in \R^{nd}$ and $\obxt \coloneqq \mathbf{1}_n \otimes \oxt \in \R^{nd}$ and let $\bW \coloneqq W \otimes I_d \in \R^{nd \times nd}$, $\bJ \coloneqq J \otimes I_d \in \R^{nd \times nd}$, where $\otimes$ denotes the Kronecker product. We next prove Lemma \ref{lm:consensus-crude}.

\begin{proof}[Proof of Lemma \ref{lm:consensus-crude}]
    Starting from the update rule \eqref{eq:gt-model-update}, it then follows that
    \begin{align*}
        \|\bxtp - \obxtp\| &= \|(\bW - \bJ)(\bxt - \alpha\wbgt)\| = \|(\bW - \bJ)(\bxt - \obxt - \alpha\wbgt)\| \\
        &\leq \lambda\|\bxt - \obxt - \alpha \wbgt\| \leq \lambda\|\bxt - \obxt\| + \alpha\lambda\|\wbgt\| \leq \lambda\|\bxt - \obxt\| + \alpha\gamma\lambda\sqrt{n}, 
    \end{align*} where the last inequality follows from the fact that $\|\wbgt\| \leq \gamma\sqrt{n}$. Iterating the inequality recursively, we get
    \begin{align*}
        \|\bxtp - \obxtp\| \leq \lambda^t\|\bx^1 - \obx^1\| + \alpha\gamma\sqrt{n}\sum_{k = 1}^{t-1}\lambda^{t-k} \stackrel{(i)}{=} \alpha\gamma\lambda\sqrt{n}\sum_{k = 1}^{t-1}\lambda^{t-1-k} \leq \frac{\alpha\gamma\lambda\sqrt{n}}{1-\lambda}, 
    \end{align*} with $(i)$ following from the shared initialization. Next, noting that $\|\xit - \oxt\|^2 \leq \sum_{i \in [n]}\|\xit - \oxt\|^2 = \|\bxt - \obxt\|^2$, it then follows that, for any $i \in [n]$ and $t \geq 1$, we have
    \begin{equation*}
        \|\xitp - \oxtp\| \leq \|\bxtp - \obxtp\| \leq \frac{\alpha\gamma\lambda\sqrt{n}}{1-\lambda},
    \end{equation*} completing the proof.
\end{proof}

Let $\nbft \coloneqq \text{col}(\nabla f_1(x_1^t),\ldots,\nabla f_n(x_n^t))$ be the column stacking of agents' full gradients evaluated at their local models at time $t$. We are now ready to prove Lemma \ref{lm:consensus-bound}.

\begin{proof}[Proof of Lemma \ref{lm:consensus-bound}]
    Starting from the update rule \eqref{eq:gt-model-update} and defining $X^t \coloneqq \frac{\|\bxt - \obxt\|^2}{n}$, it follows that for any $t \geq 1$, we have
    \begin{align*}
        X^{t+1} \leq \frac{1+\lambda^2}{2}X^t + \frac{2\alpha^2\lambda^2}{n(1-\lambda^2)}\|\wbgt\|^2 \leq \frac{1+\lambda^2}{2}X^t + \frac{4\alpha^2\lambda^2}{n(1-\lambda^2)}\big(\|\wbgt - \nbft\|^2 + \|\nbft\|^2\big), 
    \end{align*} where $(i)$ follows from Young's inequality, i.e., $(a+b)^2 \leq (1+\theta)a^2 + (1+\theta^{-1})b^2$ with $\theta = \frac{1-\lambda^2}{2\lambda^2}$, while in $(ii)$ we again used Young's inequality with $\theta = 1$. Iterating the inequality recursively, denoting by $Y^t \coloneqq \|\wbgt - \nbft\|^2 + \|\nbft\|^2$ and noting that $X^1 = 0$, we get
    \begin{align}\label{eq:cons-gap-bdd}
        X^{t+1} \leq \frac{4\alpha^2\lambda^2}{n(1-\lambda^2)}\sum_{k = 1}^{t}\bigg(\frac{1+\lambda^2}{2}\bigg)^{t-k}Y^k. 
    \end{align} We next proceed to bound $Y^k$. To that end, we have
    \begin{align*}
        Y^k &= \|\wbgk - \nbfk\|^2 + \|\nbfk\|^2 = \sum_{i \in [n]}\big(\|\tgik - \nabla f_i(\xik)\|^2 + \|\nabla f_i(\xik)\|^2\big) \\
        &= \sum_{i \in [n]}\big(\|\tgikb + \tgiku\|^2 + \|\nabla f_i(\xik)\|^2\big) \leq 2\sum_{i \in [n]}\big(\|\tgikb\|^2 + \|\tgiku\|^2 + A^2 + B^2\|\nabla f(\xik)\|^2\big),
    \end{align*} where the last inequality follows from Assumption \ref{asmpt:bdd-het}. Plugging the above bound into \eqref{eq:cons-gap-bdd} completes the proof.
\end{proof}

\begin{remark}\label{rmk:consensus}
    Note that Lemma \ref{lm:consensus-crude} implies the following bound on the average consensus gap
    \begin{equation*}
        \frac{1}{n}\sum_{i \in [n]}\|\xit - \oxt\|^2 \leq \frac{\alpha^2\gamma^2\lambda^2n}{(1-\lambda)^2}.
    \end{equation*} While nice and simple, the main issue in this bound stems from the coupling of the step-size and clipping radius of the same order, namely $\alpha^2\gamma^2$. From the choice of the step-size in Theorem \ref{thm:non-cvx-hp}, specifically $\alpha \leq \frac{C_2}{\gamma}n^{\frac{2p-1}{3p-2}}T^{\frac{1-p}{3p-2}}$ for a problem-related constant $C_2 > 0$, it follows that
    \begin{equation*}
        \frac{\alpha^2\gamma^2\lambda^2n}{(1-\lambda)^2} = \bigO\Bigg(\frac{n^{\frac{7p-4}{3p-2}}T^{\frac{2(1-p)}{3p-2}}}{(1-\lambda)^2} \Bigg).
    \end{equation*} While the above quantity on the right-hand side matches the decay of the leading term in Theorem \ref{thm:non-cvx-hp} with respect to $T$, it strictly grows in $n$. As such, naively using the bound from Lemma \ref{lm:consensus-crude} in the rest of the analysis would result in a term which matches the leading one in $T$ but increases in $n$, destroying the linear speed-up, and propagating the network effect to the leading term (via $(1-\lambda)^{-2}$). On the other hand, ignoring some terms for simplicity, Lemma \ref{lm:consensus-bound} results in the following bound
    \begin{equation*}
        \frac{1}{n}\sum_{i \in [n]}\|\xit - \oxt\|^2 = \bigO\bigg(\frac{\alpha^2}{n}\sum_{i \in [n]}\big(\|\tgikb\|^2 + \|\tgiku\|^2\big)\bigg). 
    \end{equation*} When agents' gradients are sufficiently small, i.e., when $\|\nabla f_i(\xit)\| \leq \frac{\gamma}{2}$, we can use Lemma \ref{lm:sadiev} to get
    \begin{equation*}
        \frac{1}{n}\sum_{i \in [n]}\|\xit - \oxt\|^2 = \bigO\Big(\alpha^2\gamma^{2-p} + \alpha^2\gamma^{2(1-p)}\Big) = \bigO\Big(\alpha^2\gamma^{2-p}\Big). 
    \end{equation*} The above bound again couples the step-size and the clipping radius, however, this time the step-size decays at a strictly faster rate, ensuring that the consensus gap decays at a strictly faster rate than the leading term. As such, Lemma \ref{lm:consensus-bound} provides a more fine-grained control of the consensus gap and is crucial to achieving improved rates. However, Lemma \ref{lm:consensus-crude} is still useful, see, e.g., the proof of Lemma \ref{lm:desirable-event} ahead.  
\end{remark}

\subsection{High-Probability Results}\label{subsup:hp}

In this subsection we establish the HP convergence of clipped \dsgd. In particular, a standard approach in the centralized setting is to construct a desirable event which holds with high probability and on which $\|\nabla f(x^t)\| \leq \frac{\gamma}{2}$, for all $t \in [T]$, e.g., \citep{sadiev2023highprobability,nguyen2023improved}. However, in the decentralized setting we need a slightly stronger condition, namely that $\|\nabla f_i(\xit)\| \leq \frac{\gamma}{2}$, for all $i \in [n]$ and all $t \in [T]$. Constructing a desirable event, using the heterogeneity bound and Lemma \ref{lm:consensus-crude}, we will show that such condition indeed holds with high probability. We begin by first stating the celebrated Freedman's inequality for martingales. 

\begin{lemma}\label{lm:freedman}
    Let $\{X_t\}_{t \in \N}$ be a martingale difference sequence such that for some $a > 0$ and all $t \in \N$ we have $|X_t| \leq a$ almost surely, and let $\sigma^2_t = \E[X_t^2\:|\: X_1,\ldots,X_{t-1}]$. Then for any $b,c > 0$ and $T \geq 1$
    \begin{equation*}
        \Prob\Big(\Big|\sum_{t \in [T]}X_t\Big| > b \text{ and } \sum_{t \in [T]}\sigma^2_t \leq c \Big) \leq 2\exp\bigg(-\frac{b^2}{2c + 2ab/3} \bigg).
    \end{equation*}
\end{lemma}

Next, define $\Delta_t \coloneqq f(\oxt) - f^\star$. Prior to constructing the desirable event, we provide another useful intermediary result, building on Lemmas \ref{lm:descent-inequality} and \ref{lm:consensus-bound}.

\begin{lemma}\label{lm:descent-sharp}
    Let Assumptions \ref{asmpt:network}-\ref{asmpt:bdd-het} hold. If the step-size satisfies $\alpha \leq \min\Big\{\frac{1}{2L}, \frac{1-\lambda^2}{8\lambda LB\sqrt{3}} \Big\}$, then for any $T \geq 1$, we have
    \begin{align*}
        \frac{\alpha}{16n}&\sum_{t \in [T]}\sum_{i \in [n]}\|\nabla f(\xit)\|^2 \leq \Delta_1 - \Delta_{T+1} - \alpha\sum_{t \in [T]}\langle \nabla f(\oxt), \hgtu\rangle + 2\alpha^2L\sum_{t \in [T]}\|\hgtu\|^2 \\
        &+ 2\alpha\sum_{t \in [T]}\|\hgtb\|^2 + \frac{12\alpha^3\lambda^2L^2A^2T}{(1-\lambda^2)^2} + \frac{12\alpha^3\lambda^2L^2}{n(1-\lambda^2)^2}\sum_{t \in [T]}\sum_{i \in [n]}\big(\|\tgitb\|^2 + \|\tgitu\|^2\big),
    \end{align*}
\end{lemma}

\begin{proof}
    Starting from Lemma \ref{lm:descent-inequality}, rearranging and summing up to time $T$, we get
    \begin{align*}
        \frac{\alpha}{4}\sum_{t \in [T]}\|\nabla f(\oxt)\|^2 &\leq \Delta_1 - \Delta_{T+1} - \alpha\sum_{t \in [T]}\langle \nabla f(\oxt), \hgtu\rangle \\ &+ 2\alpha^2L\sum_{t \in [T]}\|\hgtu\|^2 + 2\alpha\sum_{t \in [T]}\|\hgtb\|^2 + \frac{\alpha L^2}{2n}\sum_{t \in [T]}\|\bxt - \obxt\|^2. 
    \end{align*} Next, using the fact that $\|\nabla f(\oxt)\|^2 \geq \frac{1}{2n}\sum_{i \in [n]}\|\nabla f(\xit)\|^2 - \frac{L^2}{n}\|\bxt - \obxt\|^2$, it follows that
    \begin{align*}
        \frac{\alpha}{8n}&\sum_{t \in [T]}\sum_{i \in [n]}\|\nabla f(\xit)\|^2 \leq \Delta_1 - \Delta_{T+1} - \alpha\sum_{t \in [T]}\langle \nabla f(\oxt), \hgtu\rangle \\
        &+ 2\alpha^2L\sum_{t \in [T]}\|\hgtu\|^2 + 2\alpha\sum_{t \in [T]}\|\hgtb\|^2 + \frac{3\alpha L^2}{4n}\sum_{t \in [T]}\|\bxt - \obxt\|^2. 
    \end{align*} Applying Lemma \ref{lm:consensus-bound}, we have
    \begin{align*}
        \frac{\alpha}{8n}\sum_{t \in [T]}\sum_{i \in [n]}&\|\nabla f(\xit)\|^2 \leq \Delta_1 - \Delta_{T+1} - \alpha\sum_{t \in [T]}\langle \nabla f(\oxt), \hgtu\rangle + 2\alpha^2L\sum_{t \in [T]}\|\hgtu\|^2 + 2\alpha\sum_{t \in [T]}\|\hgtb\|^2 \\
        &+ \frac{6\alpha^3\lambda^2L^2}{n(1-\lambda^2)}\sum_{t \in [T]}\sum_{k = 1}^t\bigg(\frac{1+\lambda^2}{2}\bigg)^{t-k}\sum_{i \in [n]}\big(\|\tgikb\|^2 + \|\tgiku\|^2 + A^2 + B^2\|\nabla f(\xik)\|^2 \big) \\
        &\leq \Delta_1 - \Delta_{T+1} - \alpha\sum_{t \in [T]}\langle \nabla f(\oxt), \hgtu\rangle + 2\alpha^2L\sum_{t \in [T]}\|\hgtu\|^2 + 2\alpha\sum_{t \in [T]}\|\hgtb\|^2 \\
        &+ \frac{12\alpha^3\lambda^2L^2}{n(1-\lambda^2)^2}\sum_{i \in [n]}\sum_{t \in [T]}\big(\|\tgitb\|^2 + \|\tgitu\|^2 + A^2 + B^2\|\nabla f(\xit)\|^2 \big),
    \end{align*} where in the second inequality we used the fact that for any non-negative sequence $\{a_t\}_{t \in \N}$ and $\lambda \in [0,1)$
    \begin{equation*}
        \sum_{t \in [T]}\sum_{k \in [t]}a_k\lambda^{t-k} = \sum_{t \in [T]}a_t\sum_{k \in [t]}\lambda^{t-k} \leq \frac{1}{1-\lambda}\sum_{t \in [T]}a_t.
    \end{equation*} Rearranging the terms and noting that $\alpha \leq \frac{1-\lambda^2}{8\lambda LB\sqrt{3}}$ implies $\frac{\alpha}{8} - \frac{12\alpha^3\lambda^2L^2}{(1-\lambda^2)^2} \geq \frac{\alpha}{16}$, we then get
    \begin{align*}
        \frac{\alpha}{16n}&\sum_{t \in [T]}\sum_{i \in [n]}\|\nabla f(\xit)\|^2 \leq \Delta_1 - \Delta_{T+1} - \alpha\sum_{t \in [T]}\langle \nabla f(\oxt), \hgtu\rangle + 2\alpha^2L\sum_{t \in [T]}\|\hgtu\|^2 \\
        &+ 2\alpha\sum_{t \in [T]}\|\hgtb\|^2 + \frac{12\alpha^3\lambda^2L^2A^2T}{(1-\lambda^2)^2} + \frac{12\alpha^3\lambda^2L^2}{n(1-\lambda^2)^2}\sum_{i \in [n]}\sum_{t \in [T]}\big(\|\tgitb\|^2 + \|\tgitu\|^2\big),
    \end{align*} which completes the proof.
\end{proof}

Lemma \ref{lm:descent-sharp} provides a descent inequality which combines the network-averaged clipped quantities, as well as the strictly local ones. We are now ready to construct the desirable event and show that it holds with high probability.

\begin{lemma}\label{lm:desirable-event}
    Let Assumptions \ref{asmpt:network}-\ref{asmpt:noise} hold and let the step-size and clipping threshold be chosen as $\gamma = \max\Big\{2A + 4B\sqrt{L\Delta_1} + \frac{2\lambda L\sqrt{n}}{1-\lambda}, (nT)^{\frac{1}{3p-2}} \Big\} $ and $\alpha = \min\Big\{\frac{1}{2L}, \frac{1-\lambda^2}{8\lambda LB\sqrt{3}}, \frac{1-\lambda^2}{\lambda L\sqrt{6}}, \frac{(1-\lambda^2)^2}{6\lambda^2L}, \linebreak \frac{(1-\lambda^2)^{2/3}\Delta_1^{1/3}}{2(\lambda LA)^{2/3}T^{1/3}\sqrt[3]{12}}, \frac{C_1}{\gamma}n^{\frac{2p-1}{3p-2}}T^{\frac{1-p}{3p-2}}, \frac{C_2}{\gamma^{2/3}}n^{\frac{p}{3(3p-2)}}T^{\frac{2(1-p)}{3(3p-2)}},\frac{C_3}{\gamma} \Big\}$, where $C_1 = \min\Big\{\frac{\Delta_1}{64\sigma^{2p}}, \linebreak \frac{\sqrt{\Delta_1}}{192\log\big(\frac{6T}{\delta}\big)\sigma^{p/2}\sqrt{2L}} \Big\}$, $C_2 \leq \frac{\Delta_1^{1/3}(1-\lambda^2)^{2/3}}{12\sigma^{p/3}\lambda^{2/3}L^{2/3}}$ and $C_3 = \min\Big\{1,\frac{3\sqrt{\Delta_1}}{64\log\big(\frac{6T}{\delta}\big)\sqrt{2L}}\Big\}$, for $\delta \in (0,1)$. For any $N \in [T+1]$, let $\Gamma_N$ be the event that, for all $t \in [N]$ 
    \begin{align*}
        &- \alpha\sum_{k \in [t-1]}\langle \nabla f(\oxk), \hgku\rangle + 2\alpha^2L\sum_{k \in [t-1]}\|\hgku\|^2 + 2\alpha\sum_{k \in [t-1]}\|\hgkb\|^2 \\
        &+ \frac{12\alpha^3\lambda^2L^2A^2T}{(1-\lambda^2)^2} + \frac{12\alpha^3\lambda^2L^2}{n(1-\lambda^2)^2}\sum_{k \in [t-1]}\sum_{i \in [n]}\big(\|\tgikb\|^2 + \|\tgiku\|^2\big) \leq \Delta_1.
    \end{align*} It then holds that $\Prob(\Gamma_N) \geq 1 - \frac{(N-1)\delta}{T}$. Moreover, on the event $\Gamma_N$ we have $\|\nabla f_i(\xit)\| \leq \frac{\gamma}{2}$, for all $i \in [n]$ and $t \in [N]$.
\end{lemma}

\begin{proof}
    We prove the result by induction on $N$. For $N = 1$ we have $\Delta_1 \geq 0$, which holds trivially. Next, assume that for some $N \in [T]$, we have $\Prob(\Gamma_N) \geq 1 - \frac{(N-1)\delta}{T}$. Using Lemma \ref{lm:descent-sharp} and rearranging, it follows that under $\Gamma_N$, we have for any $t \in [N]$
    \begin{align}\label{eq:good-event-property}
        \Delta_{t} &\leq \Delta_1 - \alpha\sum_{k \in [t-1]}\langle \nabla f(\oxk), \hgku\rangle + 2\alpha^2L\sum_{k \in [t-1]}\|\hgku\|^2 + 2\alpha\sum_{k \in [t-1]}\|\hgkb\|^2 \nonumber \\
        &+ \frac{12\alpha^3\lambda^2L^2A^2T}{(1-\lambda^2)^2} + \frac{12\alpha^3\lambda^2L^2}{n(1-\lambda^2)^2}\sum_{k \in [t-1]}\sum_{i \in [n]}\big(\|\tgikb\|^2 + \|\tgiku\|^2\big) \leq 2\Delta_1.
    \end{align} From the smoothness of $f$ and \eqref{eq:good-event-property}, it follows that for all $t \in [N]$ on $\Gamma_N$ we have
    \begin{equation}\label{eq:global-grad-bdd}
        \|\nabla f(\oxt)\| \leq \sqrt{2L\Delta_t} \leq 2\sqrt{L\Delta_1}.
    \end{equation} Using the heterogeneity bound, it can then be seen that for any $i \in [n]$ and $t \in [N]$
    \begin{align}\label{eq:local-grad-bdd}
        \|\nabla f_i(\xit)\| &\leq L\|\xit - \oxt\| + \|\nabla f_i(\oxt)\| \leq L\|\bxt - \obxt\| + A + B\|\nabla f(\oxt)\| \nonumber \\
        &\stackrel{(i)}{\leq} \frac{\alpha\gamma\lambda L\sqrt{n}}{1-\lambda} + A + 2B\sqrt{L\Delta_1} \stackrel{(ii)}{\leq} \frac{\lambda L\sqrt{n}}{1-\lambda} + A + 2B\sqrt{L\Delta_1} \stackrel{(iii)}{\leq} \frac{\gamma}{2},
    \end{align} where $(i)$ follows from Lemma \ref{lm:consensus-crude}, $(ii)$ follows from the choice of step-size, while $(iii)$ follows from the choice of clipping threshold. We will now show that $\Gamma_{N+1}$ holds with probability at least $1 - \frac{N\delta}{T}$. In particular, we will show that $\Gamma_N \cap A \subseteq \Gamma_{N+1}$ for some event $A$, allowing us to focus on the part of the probability space where $\Gamma_N$ holds. Using the bound in Lemma \ref{lm:descent-sharp} with $T = N$ and the fact that the left-hand side is non-negative, we get
    \begin{align*}
        &\quad\quad\Delta_{N+1} \leq \Delta_1 - \alpha\sum_{t \in [N]}\langle \nabla f(\oxt), \hgtu\rangle + 2\alpha^2L\sum_{t \in [N]}\|\hgtu\|^2 + 2\alpha\sum_{t \in [N]}\|\hgtb\|^2 \\
        &+ \frac{12\alpha^3\lambda^2L^2A^2T}{(1-\lambda^2)^2} + \frac{12\alpha^3\lambda^2L^2}{n(1-\lambda^2)^2}\sum_{t \in [N]}\sum_{i \in [n]}\|\tgitb\|^2 + \frac{12\alpha^3\lambda^2L^2}{n(1-\lambda^2)^2}\sum_{t \in [N]}\sum_{i \in [n]}\|\tgitu\|^2.
    \end{align*} Next, we construct martingale difference sequences by adding and subtracting $\E_t\|\hgtu\|^2$ and $\E_t\|\tgitu\|^2$. For ease of notation, define $\theta^t \coloneqq \|\hgtu\|^2 - \E_t\|\hgtu\|^2$ and $\eta^t_i \coloneqq \|\tgitu\|^2 - \E_t\|\tgitu\|^2$. We then have
    \begin{equation}\label{eq:induction-step-1}
    \begin{aligned}
        &\Delta_{N+1} \leq \Delta_1 - \alpha\sum_{t \in [N]}\langle \nabla f(\oxt), \hgtu\rangle + 2\alpha^2L\sum_{t \in [N]}\theta^t + 2\alpha^2L\sum_{t \in [N]}\E_t\|\hgtu\|^2 + 2\alpha\sum_{t \in [N]}\|\hgtb\|^2 \\
        &+ \frac{12\alpha^3\lambda^2L^2A^2T}{(1-\lambda^2)^2} + \frac{12\alpha^3\lambda^2L^2}{n(1-\lambda^2)^2}\sum_{t \in [N]}\sum_{i \in [n]}\|\tgitb\|^2 + \frac{12\alpha^3\lambda^2L^2}{n(1-\lambda^2)^2}\sum_{t \in [N]}\sum_{i \in [n]}\big(\eta^t_i + \E_t\|\tgitu\|^2\big).
    \end{aligned}
    \end{equation} Recalling that we are focused on the event where $\Gamma_N$ holds and using \eqref{eq:local-grad-bdd}, we can then apply Lemmas \ref{lm:noise-behavior} and \ref{lm:sadiev} to bound some of the terms in \eqref{eq:induction-step-1}, as follows
    \begin{equation*}
        \|\tgitb\|^2 \leq \frac{4\sigma^{2p}}{\gamma^{2(p-1)}}, \:\: \|\hgtb\|^2 \leq \frac{4\sigma^{2p}}{\gamma^{2(p-1)}}, \:\: \E_t\|\tgitu\|^2 \leq 18\gamma^{2-p}\sigma^p \: \text{ and } \E_t\|\hgtu\|^2 \leq \frac{18\gamma^{2-p}\sigma^p}{n}.  
    \end{equation*} Plugging back into \eqref{eq:induction-step-1} and recalling that $N \leq T$, we get
    \begin{align}\label{eq:induction-step-2}
        &\Delta_{N+1} \leq \Delta_1 - \alpha\sum_{t \in [N]}\langle \nabla f(\oxt), \hgtu\rangle + 2\alpha^2L\sum_{t \in [N]}\theta^t + \frac{36\alpha^2\gamma^{2-p}\sigma^pLT}{n} + \frac{8\alpha\sigma^{2p}T}{\gamma^{2(p-1)}} \nonumber \\
        &+ \frac{12\alpha^3\lambda^2L^2A^2T}{(1-\lambda^2)^2} + \frac{48\alpha^3\sigma^{2p}\lambda^2L^2T}{(1-\lambda^2)^2\gamma^{2(p-1)}} + \frac{216\alpha^3\gamma^{2-p}\sigma^p\lambda^2L^2T}{(1-\lambda^2)^2} + \frac{12\alpha^3\lambda^2L^2}{n(1-\lambda^2)^2}\sum_{t \in [N]}\sum_{i \in [n]}\eta^t_i.
    \end{align} We now focus on bounding the remaining three terms in \eqref{eq:induction-step-2}, i.e., $B_1 = \alpha\sum_{t \in [N]}\langle \nabla f(\oxt), \hgtu\rangle$, $B_2 = 2\alpha^2L\sum_{t \in [N]}\theta^t$ and $B_3 = \frac{12\alpha^3\lambda^2L^2}{n(1-\lambda^2)^2}\sum_{t \in [N]}\sum_{i \in [n]}\eta^t_i$.
    
    \paragraph{Bounding $B_1$.} We define the variable $Z^t = \begin{cases} \nabla f(\oxt), & \text{if } \Delta_t \leq 2\Delta_1 \\ 0, & \text{otherwise} \end{cases}$. From the definition of $Z^t$ it follows that $\|Z^t\| \leq 2\sqrt{L\Delta_1}$, while from the definition of $\hgtu$ it follows that $\|\hgtu\| \leq 2\gamma$, therefore $\langle Z^t,\hgtu\rangle \leq 4\gamma\sqrt{L\Delta_1}$, almost surely. Further, it can be seen that $\E_t\langle Z^t,\hgtu\rangle = 0$, making $\langle Z^t,\hgtu\rangle$ a bounded martingale difference sequence. Applying Lemma \ref{lm:freedman}, we get
    \begin{equation*}
        \Prob\Big(\Big|\sum_{t \in [N]}\alpha\langle Z^t, \hgtu \rangle\Big| > b \text{ and } \sum_{t \in [N]}\alpha^2\E_t\big[\langle Z^t, \hgtu\rangle^2\big] \leq c \Big) \leq 2\exp\bigg(-\frac{b^2}{2c + 8\alpha\gamma b\sqrt{L\Delta_1}/3} \bigg),
    \end{equation*} for any $b, c > 0$. Choosing $c = C\log\big(\frac{6T}{\delta}\big)$ for some $C > 0$ (to be specified later) and 
    \begin{equation}\label{eq:bound-B1}
        b = \log\bigg(\frac{6T}{\delta}\bigg)\Bigg(\frac{4}{3}\alpha\gamma\sqrt{L\Delta_1} + \sqrt{\frac{16}{9}\alpha^2\gamma^2L\Delta_1 + 2C}\Bigg),    
    \end{equation} results in
    \begin{equation}\label{eq:prob-bound-1}
        2\exp\bigg(-\frac{b^2}{2c + 8\alpha\gamma b\sqrt{L\Delta_1}/3} \bigg) = \frac{\delta}{3T}.
    \end{equation} Next, recalling \eqref{eq:good-event-property}, it follows that $Z^t = \nabla f(\oxt)$ for all $t \in [N]$ on $\Gamma_N$, therefore $B_1 = \alpha\sum_{t \in [N]}\langle -Z^t,\hgtu\rangle$ on $\Gamma_N$. Further, under $\Gamma_N$, it holds that
    \begin{align}\label{eq:var-ineq-1}
        \sum_{t \in [N]}\alpha^2\E_t\big[\langle -Z^t, \hgtu\rangle^2\big] &\leq \sum_{t \in [N]}\alpha^2\E_t\|Z^t\|^2\|\hgtu\|^2 \leq 4\alpha^2L\Delta_1\sum_{t \in [N]}\E_t\|\hgtu\|^2 \nonumber \\
        &\stackrel{(i)}{\leq} \frac{72\alpha^2\gamma^{2-p}\sigma^pL\Delta_1N}{n} \leq \frac{72\alpha^2\gamma^{2-p}\sigma^pL\Delta_1T}{n},
    \end{align} where $(i)$ follows from Lemma \ref{lm:noise-behavior}. Therefore, choosing $C = \frac{72\alpha^2\gamma^{2-p}\sigma^pL\Delta_1T}{n}$ and defining the event $\Xi_1 = \big\{\omega: B_1(\omega) \leq b \big\}$, where $b$ is given in \eqref{eq:bound-B1}, it follows that
    \begin{align*}
        \Prob(\Gamma_N \cap \Xi_1) &= \Prob\Big( \Gamma_N \cap \Big\{\sum_{t \in [N]}\alpha\langle-Z^t,\hgtu\rangle \leq b\Big\} \Big) \\ 
        &\stackrel{(a)}{=} \Prob\bigg( \Gamma_N \cap \bigg\{\sum_{t \in [N]}\alpha\langle-Z^t,\hgtu\rangle \leq b \:\text{ or } \sum_{t \in [N]}\alpha^2\langle -Z^t,\hgtu\rangle^2 > C\log\bigg(\frac{6T}{\delta}\bigg) \bigg\} \bigg) \\
        &\stackrel{(b)}{\geq} 1 - \Prob(\Gamma_N^c) - \Prob\bigg(\sum_{t \in [N]}\alpha\langle-Z^t,\hgtu\rangle > b \: \text{ and } \sum_{t \in [N]}\alpha^2\langle -Z^t,\hgtu\rangle^2 \leq C\log\bigg(\frac{6T}{\delta}\bigg) \bigg\}\bigg) \\
        &\stackrel{(c)}{\geq} 1 - \frac{(N-1)\delta}{T} - \frac{\delta}{3T} = 1 - \frac{(N-2/3)\delta}{T}.
    \end{align*} where $(a)$ follows from \eqref{eq:var-ineq-1} and the choice of $C$, in $(b)$ we used $\Prob(A\cap B) \geq 1 - \Prob(A^c) - \Prob(B^c)$, while $(c)$ follows from \eqref{eq:prob-bound-1} and the induction hypothesis. 
    
    \paragraph{Bounding $B_2$.} Recall that $B_2 = 2\alpha^2L\sum_{t \in [N]}\theta^t$, where $\theta^t = \|\hgtu\|^2 - \E_t\|\hgtu\|^2$. By definition, we know that $\E_t[\theta^t] = 0$ and $|2\alpha^2L\theta^t| \leq 16\alpha^2\gamma^2L$, hence we can apply Lemma \ref{lm:freedman}, to get
    \begin{equation*}
        \Prob\Big(\big| B_2 \big| > d \text{ and } \sum_{t \in [N]}4\alpha^4L^2\E_t\big[(\theta^t)^2\big] \leq e \Big) \leq 2\exp\bigg(-\frac{d^2}{2e + 32\alpha^2\gamma^2Ld/3} \bigg),
    \end{equation*} for any $d, e > 0$. Choosing $e = E\log\big(\frac{6T}{\delta}\big)$ for some $E > 0$ (to be specified later) and  
    \begin{equation}\label{eq:bound-B2}
        d = \log\bigg(\frac{6T}{\delta}\bigg)\Bigg(\frac{16}{3}\alpha^2\gamma^2L + \sqrt{\frac{256}{9}\alpha^4\gamma^4L^2 + 2E}\Bigg),    
    \end{equation} results in $2\exp\Big(-\frac{d^2}{2e + 32\alpha^2\gamma^2Ld/3} \Big) = \frac{\delta}{3T}$. Further, under $\Gamma_N$, it holds that
    \begin{align*}
        \sum_{t \in [N]}\E_t\big[(2\alpha^2L\theta^t)^2\big] &\stackrel{(i)}{\leq} 16\alpha^2\gamma^2L\sum_{t \in [N]}\E_t|2\alpha^2L\theta^t| \stackrel{(ii)}{\leq} 32\alpha^4\gamma^2L^2\sum_{t \in [N]}\E_t\|\hgtu\|^2 \\
        &\leq \frac{576\alpha^4\gamma^{4-p}\sigma^pL^2N}{n} \leq \frac{576\alpha^4\gamma^{4-p}\sigma^pL^2T}{n},
    \end{align*} where $(i)$ follows from $|2\alpha^2L\theta^t| \leq 16\alpha^2\gamma^2L$, in $(ii)$ we used $\E|X-\E X| \leq 2\E X$ for a non-negative random variable, while $(iii)$ follows from Lemma \ref{lm:noise-behavior}. Therefore, choosing $E = \frac{576\alpha^4\gamma^{4-p}\sigma^pL^2T}{n}$, defining the event $\Xi_2 = \big\{\omega: B_2(\omega) \leq d \big\}$ where $d$ is given in \eqref{eq:bound-B2} and following similar steps as in the previous case, it can be shown that $\Prob(\Gamma_N \cap \Xi_1 \cap \Xi_2) \geq 1 - \frac{(N-1/3)\delta}{T}$.
    
    \paragraph{Bounding $B_3$.} Finally, consider the term $B_3 = \frac{12\alpha^3\lambda^2L^2}{n(1-\lambda^2)^2}\sum_{t \in [N]}\sum_{i \in [n]}\eta^t_i$, where $\eta^t_i = \|\tgitu\|^2 - \E_t\|\tgitu\|^2$. By definition, we know that $\E_t[\eta^t_i] = 0$ and $\Big|\frac{12\alpha^3\lambda^2L^2}{n(1-\lambda^2)^2}\eta^t_i\Big| \leq \frac{96\alpha^3\gamma^2\lambda^2L^2}{n(1-\lambda^2)^2}$, hence we can apply Lemma \ref{lm:freedman}, to get
    \begin{equation*}
        \Prob\bigg(\big|B_3 \big| > g \text{ and } \sum_{t \in [N]}\E_t\bigg(\frac{12\alpha^3\lambda^2L^2}{n(1-\lambda^2)^2}\sum_{i \in [n]}\eta^t_i\bigg)^2 \leq h \bigg) \leq 2\exp\bigg(-\frac{g^2}{2h + 192\alpha^3\gamma^2\lambda^2L^2g/3n(1-\lambda^2)^2} \bigg),
    \end{equation*} for any $g,h > 0$. Moreover, from Assumption \ref{asmpt:noise} we can conclude that $\{\eta^t_i\}_{i \in [n]}$ are independent conditioned on $\Ft$. Combined with $\E_t[\eta^t_i] = 0$, it follows that on $\Gamma_N$
    \begin{align*}
        &\sum_{t \in [N]}\E_t\bigg(\frac{12\alpha^3\lambda^2L^2}{n(1-\lambda^2)^2}\sum_{i \in [n]}\eta^t_i\bigg)^2 = \frac{144\alpha^6\lambda^4L^4}{n^2(1-\lambda^2)^4}\sum_{t \in [N]}\sum_{i \in [n]}\E_t\big[(\eta^t_i)^2\big] \\
        &\stackrel{(i)}{\leq} \frac{1152\alpha^6\gamma^2\lambda^4L^4}{n^2(1-\lambda^2)^4}\sum_{t \in [N]}\sum_{i \in [n]}\E_t\big[|\eta^t_i|\big] \stackrel{(ii)}{\leq} \frac{2304\alpha^6\gamma^2\lambda^4L^4}{n^2(1-\lambda^2)^4}\sum_{t \in [N]}\sum_{i \in [n]}\E_t\|\tgitu\|^2 \\
        &\quad\quad\quad\quad\quad \stackrel{(iii)}{\leq} \frac{41472\alpha^6\gamma^{4-p}\sigma^p\lambda^4L^4N}{n(1-\lambda^2)^4} \leq \frac{41472\alpha^6\gamma^{4-p}\sigma^p\lambda^4L^4T}{n(1-\lambda^2)^4},
    \end{align*} where in $(i)$ we used the fact that $|\eta^t_i| \leq 8\gamma^2$, $(ii)$ follows from $\E|X - \E X| \leq 2\E X$ for $X \geq 0$, while in $(iii)$ we used Lemma \ref{lm:sadiev}. Following the same reasoning as in the previous two cases, it can be seen that choosing $h = H\log\big(\frac{6T}{\delta}\big)$, where $H = \frac{41472\alpha^6\gamma^{4-p}\sigma^p\lambda^4L^4T}{n(1-\lambda^2)^4}$ and  
    \begin{equation}\label{eq:bound-B3}
        g = \log\bigg(\frac{6T}{\delta}\bigg)\Bigg(\frac{32\alpha^3\gamma^2\lambda^2L^2}{n(1-\lambda^2)^2} + \sqrt{\frac{1024\alpha^6\gamma^4\lambda^4L^4}{n^2(1-\lambda^2)^4} + 2H}\Bigg), 
    \end{equation} results in $2\exp\Big(-\frac{g^2}{2h + 192\alpha^3\gamma^2\lambda^2L^2g/3n(1-\lambda^2)^2} \Big) = \frac{\delta}{3T}$. Therefore, defining $\Xi_3 = \{\omega: B_3(\omega) \leq g\}$ and following similar reasoning as before, it can be concluded that
    \begin{equation}\label{eq:desired-probability-bound}
        \Prob(\Gamma_N \cap \Xi_1 \cap \Xi_2 \cap \Xi_3) \geq 1 - \frac{N\delta}{T}.
    \end{equation}
    
    \paragraph{Completing the induction.} Consider the event $\Gamma_N \cap \Xi_1 \cap \Xi_2 \cap \Xi_3$. From \eqref{eq:desired-probability-bound} we know that $\Prob(\Gamma_N \cap \Xi_1 \cap \Xi_2 \cap \Xi_3) \geq 1 - \frac{N\delta}{T}$, while from \eqref{eq:induction-step-2} we know that on this event 
    \begin{align}\label{eq:induction-step-3}
        &\Delta_{N+1} \leq \Delta_1 + b + d + g + \frac{36\alpha^2\gamma^{2-p}\sigma^pLT}{n} + \frac{8\alpha\sigma^{2p}T}{\gamma^{2(p-1)}} \nonumber \\
        &+ \frac{12\alpha^3\lambda^2L^2A^2T}{(1-\lambda^2)^2} + \frac{48\alpha^3\sigma^{2p}\lambda^2L^2T}{(1-\lambda^2)^2\gamma^{2(p-1)}} + \frac{216\alpha^3\gamma^{2-p}\sigma^p\lambda^2L^2T}{(1-\lambda^2)^2},
    \end{align} where $b,d,g$ are given in \eqref{eq:bound-B1}, \eqref{eq:bound-B2} and \eqref{eq:bound-B3}. Note that if we can choose $\alpha$ and $\gamma$ such that the right-hand side in \eqref{eq:induction-step-3} is less than or equal to $2\Delta_1$, it would then follow that on $\Gamma_N \cap \Xi_1 \cap \Xi_2 \cap \Xi_3$ we have $\Delta_{N+1} \leq 2\Delta_1$, hence $\Gamma_N \cap \Xi_1 \cap \Xi_2 \cap \Xi_3 \subseteq \Gamma_{N+1}$ and $\Prob(\Gamma_{N+1}) \geq \Prob(\Gamma_N \cap \Xi_1 \cap \Xi_2 \cap \Xi_3) \geq 1 - \frac{N\delta}{T}$, completing the induction. Therefore, our aim is to show that there exist $\alpha,\gamma$ such that
    \begin{align}\label{eq:induction-step-4}
        b &+ d + g + \frac{36\alpha^2\gamma^{2-p}\sigma^pLT}{n} + \frac{8\alpha\sigma^{2p}T}{\gamma^{2(p-1)}} + \frac{12\alpha^3\lambda^2L^2A^2T}{(1-\lambda^2)^2} \nonumber \\
        &+ \frac{48\alpha^3\sigma^{2p}\lambda^2L^2T}{(1-\lambda^2)^2\gamma^{2(p-1)}} + \frac{216\alpha^3\gamma^{2-p}\sigma^p\lambda^2L^2T}{(1-\lambda^2)^2} \leq \Delta_1,
    \end{align} where $b,d,g$ are given in \eqref{eq:bound-B1}, \eqref{eq:bound-B2} and \eqref{eq:bound-B3}. Let the clipping threshold and step-size satisfy $\gamma \geq \max\Big\{2A + 4B\sqrt{L\Delta_1} + \frac{2\lambda L\sqrt{n}}{1-\lambda}, (nT)^{\frac{1}{3p-2}} \Big\}$ and 
    \begin{equation*}
        \alpha \leq \min\bigg\{C_1n^{\frac{2p-1}{3p-2}}T^{\frac{1-p}{3p-2}}\gamma^{-1}, C_2n^{\frac{p}{3(3p-2)}}T^{\frac{2(1-p)}{3(3p-2)}}\gamma^{-\frac{2}{3}}, \frac{1}{2L}, \frac{1-\lambda^2}{8\lambda LB\sqrt{3}} \bigg\},
    \end{equation*} for some $C_1,C_2 > 0$. It can then be verified that $\alpha^2\gamma^{2-p} \leq nC_1^{2}T^{-1}$, $\alpha\gamma^{2(1-p)} \leq C_1T^{-1}$ and $\alpha^3\gamma^{2-p} \leq C_2^3T^{-1}$, hence choosing $C_1 \leq \min\Big\{\frac{\sqrt{\Delta_1}}{12\sigma^{\nicefrac{p}{2}}\sqrt{2L}}, \frac{\Delta_1}{64\sigma^{2p}} \Big\}$, $C_2 \leq \frac{\Delta_1^{1/3}(1-\lambda^2)^{2/3}}{12\sigma^{p/3}\lambda^{2/3}L^{2/3}}$ and $\alpha \leq \min\bigg\{\sqrt[3]{\frac{(1-\lambda^2)^2\Delta_1}{96\lambda^2L^2A^2T}},\frac{1-\lambda^2}{\lambda L\sqrt{6}}\bigg\}$, it can be verified that \eqref{eq:induction-step-4} becomes $b + d + g \leq \frac{3\Delta_1}{8}$. We now bound each of these terms. First, recall that
    \begin{equation*}
        b = \log\bigg(\frac{6T}{\delta}\bigg)\Bigg(\frac{4}{3}\alpha\gamma\sqrt{L\Delta_1} + \sqrt{\frac{16}{9}\alpha^2\gamma^2L\Delta_1 + 2C}\Bigg),    
    \end{equation*} where $C = \frac{72\alpha^2\gamma^{2-p}\sigma^pL\Delta_1T}{n} \leq 72C_1^2\sigma^pL\Delta_1$. Further, if $\alpha \leq C_3\gamma^{-1}$, where $C_3 \leq \frac{3\sqrt{\Delta_1}}{64\log(\frac{6T}{\delta})\sqrt{2L}}$ and $C_1 \leq \frac{\sqrt{\Delta_1}}{192\log(\frac{6T}{\delta})\sigma^{p/2}\sqrt{2L}}$, it follows that
    \begin{equation}\label{eq:bound-1}
        b \leq \log\bigg(\frac{6T}{\delta}\bigg)\Bigg(\frac{4}{3}C_3\sqrt{L\Delta_1} + \sqrt{\frac{16}{9}C_3^2L\Delta_1 + 144C_1^2\sigma^pL\Delta_1}\Bigg) \leq \frac{\Delta_1}{8}.
    \end{equation} Next, recall that 
     \begin{equation*}
        d = \log\bigg(\frac{6T}{\delta}\bigg)\Bigg(\frac{16}{3}\alpha^2\gamma^2L + \sqrt{\frac{256}{9}\alpha^4\gamma^4L^2 + 2E}\Bigg),    
    \end{equation*} where $E = \frac{576\alpha^4\gamma^{4-p}\sigma^pL^2T}{n} \leq 576C_1^2C_3^2\sigma^pL^2 \leq \frac{288C_3^2L\Delta_1}{192^2\log^2(\frac{6T}{\delta})}$. From the previous choice of $C_3$, it can be verified that
    \begin{equation}\label{eq:bound-2}
        d \leq \log\bigg(\frac{6T}{\delta}\bigg)\Bigg(\frac{16}{3}C_3^2L + \sqrt{\frac{256}{9}C_3^4L^2 + \frac{C_3^2L\Delta_1}{64\log^2(\frac{6T}{\delta})}}\Bigg) \leq \frac{\Delta_1}{8}.
    \end{equation} Finally, recall that   
    \begin{equation*}
        g = \log\bigg(\frac{6T}{\delta}\bigg)\Bigg(\frac{32\alpha^3\gamma^2\lambda^2L^2}{n(1-\lambda^2)^2} + \sqrt{\frac{1024\alpha^6\gamma^4\lambda^4L^4}{n^2(1-\lambda^2)^4} + 2H}\Bigg), 
    \end{equation*} where $H = \frac{41472\alpha^6\gamma^{4-p}\sigma^p\lambda^4L^4T}{n(1-\lambda^2)^4} \leq \frac{41472\alpha^2C_1^2C_3^2\sigma^p\lambda^4L^4}{(1-\lambda^2)^4} \leq \frac{9\Delta_1\alpha^2C_3^2\lambda^4L^3}{16(1-\lambda^2)^4\log^2(\frac{6T}{\delta})} \leq \frac{\Delta_1C_3^2L}{64\log^2(\frac{6T}{\delta})}$, where the last inequality follows from $\alpha \leq \frac{(1-\lambda^2)^2}{6\lambda^2L}$. Using the same step-size condition, we then have
    \begin{equation}\label{eq:bound-3}
        g \leq \log\bigg(\frac{6T}{\delta}\bigg)\Bigg(\frac{16C_3^2L}{3} + \sqrt{\frac{256}{9}C_3^4L^2 + \frac{C_3^2L\Delta_1}{32\log^2(\frac{6T}{\delta})}} \Bigg) \leq \frac{\Delta_1}{8}.
    \end{equation} Combining \eqref{eq:bound-1}, \eqref{eq:bound-2} and \eqref{eq:bound-3}, we finally get $b + d + g \leq \frac{3\Delta_1}{8}$, completing the proof.
\end{proof}

We are now ready to prove Theorem \ref{thm:non-cvx-hp}. Prior to that, we state the full version, including the higher-order terms omitted in the main body. 

{
\renewcommand{\thetheorem}{1}
\begin{theorem}
    Let Assumptions \ref{asmpt:network}-\ref{asmpt:noise} hold. If for any $T \geq 1$ and $\delta \in (0,1)$, the clipping threshold and the step-size are chosen as $\gamma = \max\Big\{2A + 4B\sqrt{L\Delta_1} + \frac{2\lambda L\sqrt{n}}{1-\lambda}, (nT)^{\frac{1}{3p-2}} \Big\}$ and $\alpha = \min\Big\{\frac{1}{2L}, \frac{1-\lambda^2}{8\lambda LB\sqrt{3}}, \frac{1-\lambda^2}{\lambda L\sqrt{6}}, \frac{(1-\lambda^2)^2}{6\lambda^2L}, \frac{(1-\lambda^2)^{2/3}\Delta_1^{1/3}}{2(\lambda LA)^{2/3}T^{1/3}\sqrt[3]{12}}, \frac{C_1}{\gamma}n^{\frac{2p-1}{3p-2}}T^{\frac{1-p}{3p-2}},\frac{C_2}{\gamma^{2/3}}n^{\frac{p}{3(3p-2)}}T^{\frac{2(1-p)}{3(3p-2)}},\frac{C_3}{\gamma} \Big\}$, where $C_1 = \min\Big\{\frac{\Delta_1}{64\sigma^{2p}}, \frac{\sqrt{\Delta_1}}{192\log\big(\frac{6T}{\delta}\big)\sigma^{p/2}\sqrt{2L}} \Big\}$, $C_2 = \frac{\Delta_1^{1/3}(1-\lambda^2)^{2/3}}{12\sigma^{p/3}\lambda^{2/3}L^{2/3}}$ and $C_3 = \min\Big\{1,\frac{3\sqrt{\Delta_1}}{64\log\big(\frac{6T}{\delta}\big)\sqrt{2L}}\Big\}$, then with probability at least $1 - \delta$
    \begin{align*}
        \frac{1}{nT}\sum_{t \in [T]}\sum_{i \in [n]}\|\nabla f(\xit)\|^2 = \widetilde{\bigO}\Bigg(\max\Bigg\{(nT)^{\frac{2(1-p)}{3p-2}}, \frac{n^{\frac{1}{3p-2}}}{T^{\frac{3(p-1)}{3p-2}}}, \frac{n^{\frac{-p}{6p-4}}T^{\frac{1-2p}{3p-2}}}{(1-\lambda)},\\ \frac{n^{\frac{2(p-1)}{3(3p-2)}}}{(1-\lambda)^{\frac{4}{3}}T^{\frac{7p-4}{3(3p-2)}}}, \frac{n^{\frac{2-p}{3(3p-2)}}}{(1-\lambda^2)^{\frac{2}{3}}T^{\frac{7p-6}{3(3p-2)}}}, \frac{T^{-2/3}}{(1-\lambda^2)^{2/3}}, \frac{T^{-1}}{(1-\lambda^2)^2}, \frac{\sqrt{n}T^{-1}}{1-\lambda} \Bigg\}\Bigg).
    \end{align*} 
\end{theorem}
}

\begin{proof}
    Starting from Lemma \ref{lm:descent-sharp}, we have
    \begin{align*}
        \frac{\alpha}{16n}&\sum_{t \in [T]}\sum_{i \in [n]}\|\nabla f(\xit)\|^2 \leq \Delta_1 - \Delta_{T+1} - \alpha\sum_{t \in [T]}\langle \nabla f(\oxt), \hgtu\rangle + 2\alpha^2L\sum_{t \in [T]}\|\hgtu\|^2 \\
        &+ 2\alpha\sum_{t \in [T]}\|\hgtb\|^2 + \frac{12\alpha^3\lambda^2L^2A^2T}{(1-\lambda^2)^2} + \frac{12\alpha^3\lambda^2L^2}{n(1-\lambda^2)^2}\sum_{i \in [n]}\sum_{t \in [T]}\big(\|\tgitb\|^2 + \|\tgitu\|^2\big).
    \end{align*} Using Lemma \ref{lm:desirable-event}, we know that, with probability at least $1 - \delta$
    \begin{align*}
        &\frac{\alpha}{16n}\sum_{t \in [T]}\sum_{i \in [n]}\|\nabla f(\xit)\|^2 \leq \Delta_1 - \Delta_{T+1} - \alpha\sum_{t \in [T]}\langle \nabla f(\oxt), \hgtu\rangle + 2\alpha^2L\sum_{t \in [T]}\|\hgtu\|^2 \\
        &+ 2\alpha\sum_{t \in [T]}\|\hgtb\|^2 + \frac{12\alpha^3\lambda^2L^2A^2T}{(1-\lambda^2)^2} + \frac{12\alpha^3\lambda^2L^2}{n(1-\lambda^2)^2}\sum_{i \in [n]}\sum_{t \in [T]}\big(\|\tgitb\|^2 + \|\tgitu\|^2\big) \leq 2\Delta_1.
    \end{align*} Dividing both sides by $\frac{\alpha T}{16}$, we get that, with probability at least $1 - \delta$
    \begin{equation*}
        \frac{1}{nT}\sum_{t \in [T]}\sum_{i \in [n]}\|\nabla f(\xit)\|^2 \leq \frac{32\Delta_1}{\alpha T}.
    \end{equation*} We now consider the different conditions on the step-size and clipping threshold. First, if $\alpha = \textit{constant}$, the above bound evaluates to (in the worst-case with respect to $n$ and $1-\lambda$)
    \begin{equation}\label{eq:bound-case-1}
        \frac{1}{nT}\sum_{t \in [T]}\sum_{i \in [n]}\|\nabla f(\xit)\|^2 = \bigO\bigg( \frac{1}{(1-\lambda^2)^2T}\bigg).
    \end{equation} Next, if $\alpha = \frac{(1-\lambda^2)^{2/3}\Delta_1^{1/3}}{2(\lambda LA)^{2/3}T^{1/3}\sqrt[3]{12}}$, then
    \begin{equation}\label{eq:bound-case-2}
        \frac{1}{nT}\sum_{t \in [T]}\sum_{i \in [n]}\|\nabla f(\xit)\|^2 = \bigO\bigg( \frac{1}{(1-\lambda^2)^{2/3}T^{2/3}}\bigg).
    \end{equation} Third, if $\alpha = C_1n^{\frac{2p-1}{3p-2}}T^{\frac{1-p}{3p-2}}\gamma^{-1}$, then 
    \begin{equation}\label{eq:bound-case-3}
        \frac{1}{nT}\sum_{t \in [T]}\sum_{i \in [n]}\|\nabla f(\xit)\|^2 = \widetilde{\bigO}\Bigg(\max\Bigg\{\frac{1}{(1-\lambda)n^{\frac{p}{6p-4}}T^{\frac{2p-1}{3p-2}}},\frac{1}{(nT)^{\frac{2(p-1)}{3p-2}}}\Bigg\}\Bigg).
    \end{equation} Fourth, if $\alpha = C_2n^{\frac{p}{3(3p-2)}}T^{\frac{2(1-p)}{3(3p-2)}}\gamma^{-2/3}$, then
    \begin{equation}\label{eq:bound-case-4}
        \frac{1}{nT}\sum_{t \in [T]}\sum_{i \in [n]}\|\nabla f(\xit)\|^2 = \widetilde{\bigO}\Bigg(\max\Bigg\{\frac{n^{\frac{2(p-1)}{3(3p-2)}}}{(1-\lambda)^{\frac{4}{3}}T^{\frac{7p-4}{3(3p-2)}}},\frac{n^{\frac{2-p}{3(3p-2)}}}{(1-\lambda^2)^{\frac{2}{3}}T^{\frac{7p-6}{3(3p-2)}}}\Bigg\}\Bigg).
    \end{equation} Finally, the only remaining case is that of $\alpha = C_3\gamma^{-1}$, where we get 
    \begin{equation}\label{eq:bound-case-5}
        \frac{1}{nT}\sum_{t \in [T]}\sum_{i \in [n]}\|\nabla f(\xit)\|^2 = \widetilde{\bigO}\Bigg(\max\Bigg\{\frac{\sqrt{n}}{(1-\lambda)T}, \frac{n^{\frac{1}{3p-2}}}{T^{\frac{3(p-1)}{3p-2}}}\Bigg\}\Bigg).
    \end{equation} The proof is completed by combining \eqref{eq:bound-case-1}-\eqref{eq:bound-case-5}.
\end{proof}

\subsection{In-Expectation Results}\label{subsup:ie}

To prove the Theorem \ref{thm:non-cvx-ie}, we build on the HP convergence result using the fact that clipped methods generate sequences with bounded increments, an idea previously used in, e.g., \citep{armacki2024_ldp+mse,chezhegov2026last}. Let $\nabla_1 \coloneqq \|\nabla f(\ox^1)\|^2$. We then have the following result. 

{
\renewcommand{\thetheorem}{2}
\begin{theorem}
    Let Assumptions \ref{asmpt:network}-\ref{asmpt:noise} hold. If for any $T \geq 2$, the clipping threshold and the step-size are chosen as $\gamma = \max\Big\{2A + 4B\sqrt{L\Delta_1} + \frac{2\lambda L\sqrt{n}}{1-\lambda}, (nT)^{\frac{1}{3p-2}} \Big\}$ and $\alpha = \min\Big\{\frac{1}{2L}, \frac{1-\lambda^2}{8\lambda LB\sqrt{3}}, \frac{1-\lambda^2}{\lambda L\sqrt{6}}, \linebreak \frac{(1-\lambda^2)^2}{6\lambda^2L}, \frac{(1-\lambda^2)^{2/3}\Delta_1^{1/3}}{2(\lambda LA)^{2/3}T^{1/3}\sqrt[3]{12}}, \frac{C_1}{\gamma}n^{\frac{2p-1}{3p-2}}T^{\frac{1-p}{3p-2}},\frac{C_2}{\gamma^{2/3}}n^{\frac{p}{3(3p-2)}}T^{\frac{2(1-p)}{3(3p-2)}},\frac{C_3}{\gamma} \Big\}$, where $C_1 = \min\Big\{\frac{\Delta_1}{64\sigma^{2p}}, \linebreak \frac{\sqrt{\Delta_1}}{192\log(6n^4T^5)\sigma^{p/2}\sqrt{2L}} \Big\}$, $C_2 = \frac{\Delta_1^{1/3}(1-\lambda^2)^{2/3}}{12\sigma^{p/3}\lambda^{2/3}L^{2/3}}$ and $C_3 = \min\Big\{1,\frac{3\sqrt{\Delta_1}}{64\log(6n^4T^5)\sqrt{2L}}\Big\}$, then
    \begin{align*}
        \frac{1}{nT}\sum_{t \in [T]}\sum_{i \in [n]}\E\|\nabla f(\xit)\|^2 = \widetilde{\bigO}\Bigg(\max\Bigg\{(nT)^{\frac{2(1-p)}{3p-2}}, \frac{n^{\frac{1}{3p-2}}}{T^{\frac{3(p-1)}{3p-2}}}, \frac{n^{\frac{-p}{6p-4}}T^{\frac{1-2p}{3p-2}}}{(1-\lambda)},\\ \frac{n^{\frac{2(p-1)}{3(3p-2)}}}{(1-\lambda)^{\frac{4}{3}}T^{\frac{7p-4}{3(3p-2)}}}, \frac{n^{\frac{2-p}{3(3p-2)}}}{(1-\lambda^2)^{\frac{2}{3}}T^{\frac{7p-6}{3(3p-2)}}}, \frac{T^{-2/3}}{(1-\lambda^2)^{2/3}}, \frac{T^{-1}}{(1-\lambda^2)^2}, \frac{\sqrt{n}T^{-1}}{1-\lambda} \Bigg\}\Bigg).
    \end{align*} 
\end{theorem}
}

\begin{proof}
    For any $T \geq 2$ and $\delta \in (0,1)$, denote by $G^T_\delta \coloneqq \frac{1}{nT}\sum_{t \in [T]}\sum_{i \in [n]}\|\nabla f(\xit(\delta))\|^2$ and let $M^T_\delta =  \wbigO\Big(\max\Big\{(nT)^{\frac{2(1-p)}{3p-2}}, \frac{n^{\frac{1}{3p-2}}}{T^{\frac{3(p-1)}{3p-2}}}, \frac{n^{\frac{-p}{6p-4}}T^{\frac{1-2p}{3p-2}}}{(1-\lambda)},\frac{n^{\frac{2(p-1)}{3(3p-2)}}}{(1-\lambda)^{\frac{4}{3}}T^{\frac{7p-4}{3(3p-2)}}}, \frac{n^{\frac{2-p}{3(3p-2)}}}{(1-\lambda^2)^{\frac{2}{3}}T^{\frac{7p-6}{3(3p-2)}}}, \frac{T^{-2/3}}{(1-\lambda^2)^{2/3}}, \linebreak \frac{T^{-1}}{(1-\lambda^2)^2}, \frac{\sqrt{n}T^{-1}}{1-\lambda} \Big\}\Big)$, noting that, for any $\delta \in (0,1)$, we have
    \begin{align}\label{eq:expectation-setup}
        \E[G^T_\delta] &= \E\big[G^T_\delta\big(\mathbb{I}_{\{G^T_\delta \leq M^T_\delta \}} + \mathbb{I}_{\{G^T_\delta > M^T_\delta \}}\big)\big] \stackrel{(i)}{\leq} \E\big[M^T_\delta\mathbb{I}_{\{G^T_\delta \leq M^T_\delta \}}\big] + \sqrt{\E\big[\big(G^T_\delta\big)^2\big]\E\big[\mathbb{I}^2_{\{G^T_\delta > M^T_\delta \}}\big]} \nonumber \\
        &= M^T_\delta\Prob\big(G^T_\delta \leq M^T_\delta \big) + \sqrt{\E\big[\big(G^T_\delta\big)^2\big]\Prob\big(G^T_\delta > M^T_\delta \big)} \stackrel{(ii)}{\leq} M^T_\delta + \sqrt{\delta\E\big[\big(G^T_\delta\big)^2\big]},
    \end{align} where $(i)$ follows from H\"{o}lder's inequality, while $(ii)$ follows from Theorem \ref{thm:non-cvx-hp}. We now proceed to bound the term $G^T_\delta$. To that end, first note that for any $t \geq 1$, it holds that
    \begin{align}\label{eq:iterate-bdd}
        \|\bxt - \bx^1\| &= \|\bW(\bxtm - \bx^1 - \alpha\tbgt)\| \leq \|\bW(\bxtm - \bx^1)\| + \alpha\|\bW\tbgt\| \nonumber \\ 
        &\leq \|\bxtm - \bx^1\| + \alpha\gamma\sqrt{n} \leq \ldots \leq \alpha\gamma(t-1)\sqrt{n} \leq \alpha\gamma t\sqrt{n},
    \end{align} where the second inequality follows from $\|\bW\| \leq 1$ and $\|\tbgt\| \leq \gamma\sqrt{n}$. From the definition of $G^T_\delta$, we then get
    \begin{equation*}
        G^T_{\delta} \leq \frac{2}{n}\sum_{i \in [n]}\|\nabla f(x_i^1)\|^2 + \frac{2L^2}{nT}\sum_{t \in [T]}\|\bxt - \bx^1\|^2 \leq 2\nabla_1 + \frac{2\alpha^2\gamma^2L^2}{T}\sum_{t \in [T]}t^2 \leq 2\nabla_1 + 2\alpha^2\gamma^2L^2T^2, 
    \end{equation*} where the second inequality follows from \eqref{eq:iterate-bdd} and the fact that agents have an identical initialization. Plugging this into \eqref{eq:expectation-setup}, we get
    \begin{equation}
        \E\big[G_{\delta}^T\big] \leq M_\delta^T + 2(\nabla_1 + \alpha^2\gamma^2L^2T^2)\delta^{1/2} \leq M_\delta^T + 2\bigg(\nabla_1 + C_1^2L^2(nT)^{\frac{2(2p-1)}{3p-2}}\bigg)\delta^{1/2},
    \end{equation} where the second inequality follows from the choice of $\alpha$ and $\gamma$. Choosing $\delta = (nT)^{-4} \in (0,1)$, it follows that 
    \begin{equation*}
        \E\big[G^T_\delta\big] \leq M_{\delta}^T + \frac{2\nabla_1}{n^2T^2} + \frac{C_1^2L^2}{(nT)^{\frac{2(p-1)}{3p-2}}} = \wbigO\big(M_\delta^T\big).
    \end{equation*}
    
\end{proof}

\section{Experiment Details}\label{sup:num}

In this section we provide details of the experimental setup and some additional results, following similar setups as the ones considered in \citep{yu2026decentralized}. Subsection \ref{subsup:quad} details the quadratic example, Subsection \ref{subsup:synth} details the robust regression experiments, while Subsection \ref{subsup:trans} details the experiment on distributed training of transformers. 

\subsection{Convergence on a Quadratic Cost}\label{subsup:quad}

As mentioned in the main body, \cite{yu2026decentralized} provide a simple quadratic example for which \dsgd with normalization fails to converge, even with noiseless gradients. We verify numerically that the non-convergence is fixed when using gradient clipping with an appropriately chosen clipping radius. In particular, we consider a network with $n = 20$ agents communicating over a fully connected graph. Each agent has access to a local quadratic cost $f_i(x) = \frac{1}{2}(x-a_i)^2$, where $a_i = 0$ for agents $i \in [10]$, while $a_i = 10$ for the rest. As such, it can be readily verified that Assumption \ref{asmpt:bdd-het} is satisfied, e.g., with $A = 5\sqrt{2}$ and $B = \sqrt{2}$.\footnote{It can be shown that Assumption \ref{asmpt:bdd-het} is satisfied for any $A,B$ such that $B > 1$ and $A^2 \geq \frac{25B^2}{B^2 - 1}$.} All agents are initialized at $x_i^1=0.5$, while the global cost is minimized at $x^\star=5$. We compare vanilla \dsgd with normalized (\dsgdn) and clipped \dsgd (\dsgdclip), with all three methods having access to noiseless gradients and using the constant step-size $\alpha=0.5$, and clipped \dsgd using the clipping radius $\gamma=6$. Convergence of all three methods is presented in Figure \ref{fig:normalization_clipping_example} in the main body, while in Figure \ref{fig:normalization_clipping_example_add} we additionally show the magnitude of the effective gradient, i.e., $\left|\frac{1}{n}\sum_{i=1}^{n}\widetilde g_i^t\right|,$ where $\widetilde g_i^t$ denotes the vanilla, normalized or clipped gradient used in the update. \dsgdn cancels out the effective gradient (Figure \ref{fig:normalization_clipping_example_add}, right) and remains stuck at initialization, while \dsgdclip preserves an effective direction, even when clipping is active in the first few iterations, gradually converging to the optimum.

\begin{figure}
    \centering
    \begin{tabular}{cc}
         \includegraphics[width=0.45\linewidth]{Plots/example1_error.png}
         &
         \includegraphics[width=0.45\linewidth]{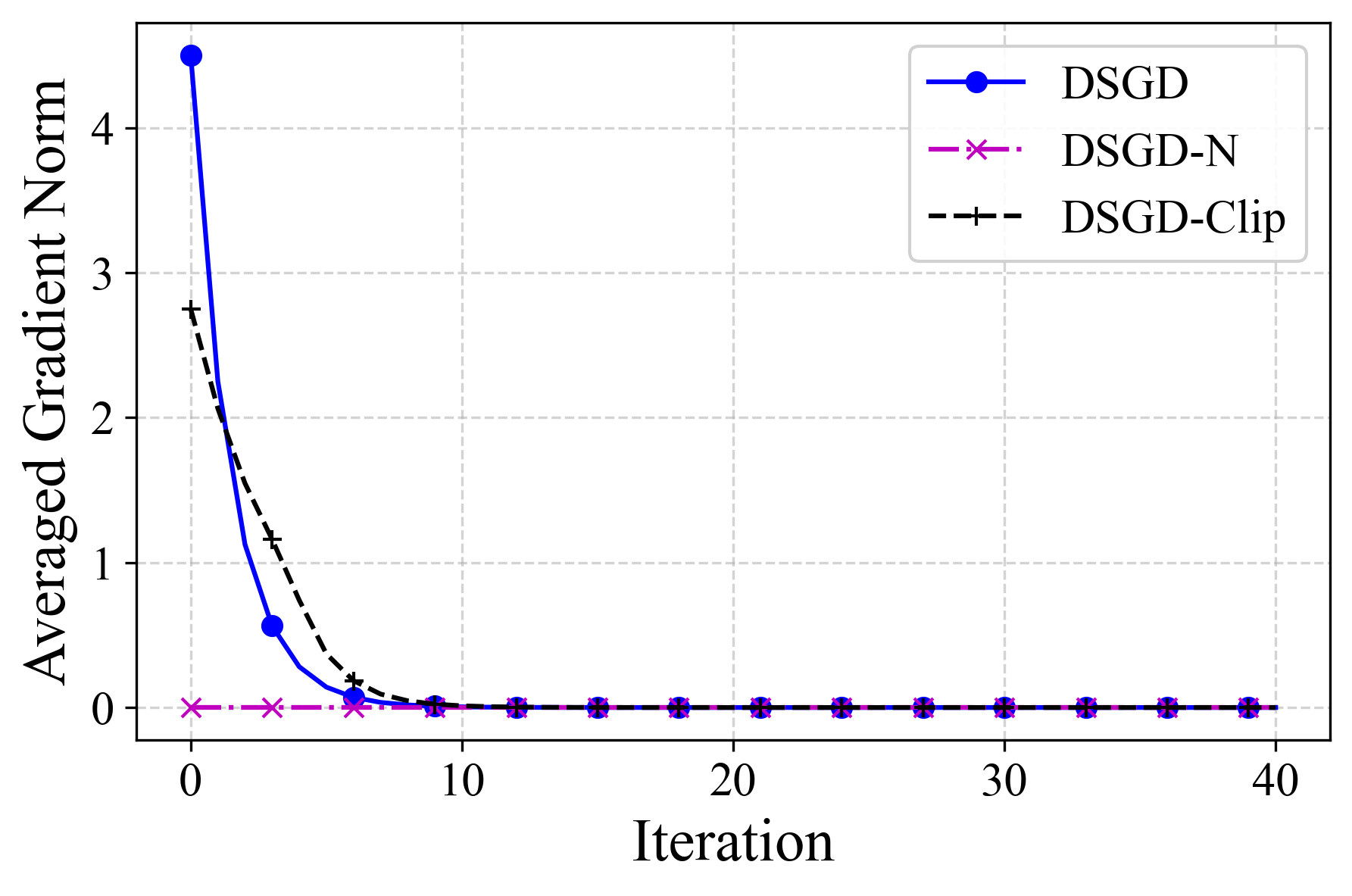} 
    \end{tabular}
    
    \caption{Convergence of vanilla, normalized and clipped \dsgd on a quadratic example.}
    \label{fig:normalization_clipping_example_add}
\end{figure}

\subsection{Decentralized Linear Regression}\label{subsup:synth}

Agents have access to a local dataset $\{(A_{i,k},b_{i,k})\}_{k=1}^{N_i}$ and the local cost is given by $f_i(x) = \frac{1}{N_i}\sum_{k=1}^{N_i}\ell\big(b_{i,k}-\langle A_{i,k},x\rangle\big)$, where $\ell$ is the Tukey biweight loss, defined as
\begin{align}
    \ell(r)=
    \begin{cases}
        \displaystyle
        \frac{c^2}{6}
        \left[1-\left(1-\frac{r^2}{c^2}\right)^3\right],
        & |r|\leq c,\\[1ex]
        \displaystyle \frac{c^2}{6},
        & |r|>c.
    \end{cases}
\end{align} We choose the constant $c=4.6851$. We tune every method (i.e., vanilla, normalized, clipped and component-wise clipped \dsgd) separately for each noise distribution. All methods choose their step-size $\alpha$ over the grid
$$
\alpha \in \{10^{-5},\,5{\times}10^{-5},\,10^{-4},\,5{\times}10^{-4},
\,10^{-3},\,5{\times}10^{-3},\,10^{-2},\,5{\times}10^{-2},
\,10^{-1},\,5{\times}10^{-1},\,1,\,5,\,10\}.
$$
For \dsgdclip and \dsgdcclip, we further chose the clipping radius $\gamma$ over the grid
$$
\gamma \in \{10^{-3},\,5{\times}10^{-3},\,10^{-2},\,5{\times}10^{-2},
\,10^{-1},\,5{\times}10^{-1},\,1,\,5,\,10,\,20,\,50,\,100,\,500\}.
$$
For linear speed-up experiments, we use the best step-size and clipping radius selected in the preceding experiment, and repeat the experiment over $30$ runs with independently sampled graphs and gradient noise in each run.

\subsection{Decentralized Training of Transformers}\label{subsup:trans}

For each method and dataset, we chose the step-size from the grid
$$
\alpha\in \{10^{-4},5\times10^{-4},10^{-3},5\times10^{-3},
10^{-2},5\times10^{-2},0.1,0.5,0.8,1,2,5,10,20,50,100\},
$$
while for \dsgdclip and \dsgdcclip, we additionally chose the clipping radius via a grid search over
$$
\gamma\in\{10^{-4},10^{-3},10^{-2},0.05,0.1,0.5,1,5,10,100,500\}.
$$
The best hyperparameter combination is selected based on the final validation loss, measured as log-perplexity. We then run each method independently five times and report the average validation performance.

Figure \ref{fig:ring-curves} shows the full learning curves corresponding to the results in Table \ref{tab:transformers} in the main body. We can see that \dsgdcclip consistently achieves the lowest validation loss on both datasets, while \dsgdn and \dsgdclip perform similarly and substantially improve over vanilla \dsgd.

\begin{figure}
\centering
\begin{tabular}{cc}
\includegraphics[width=0.45\linewidth]{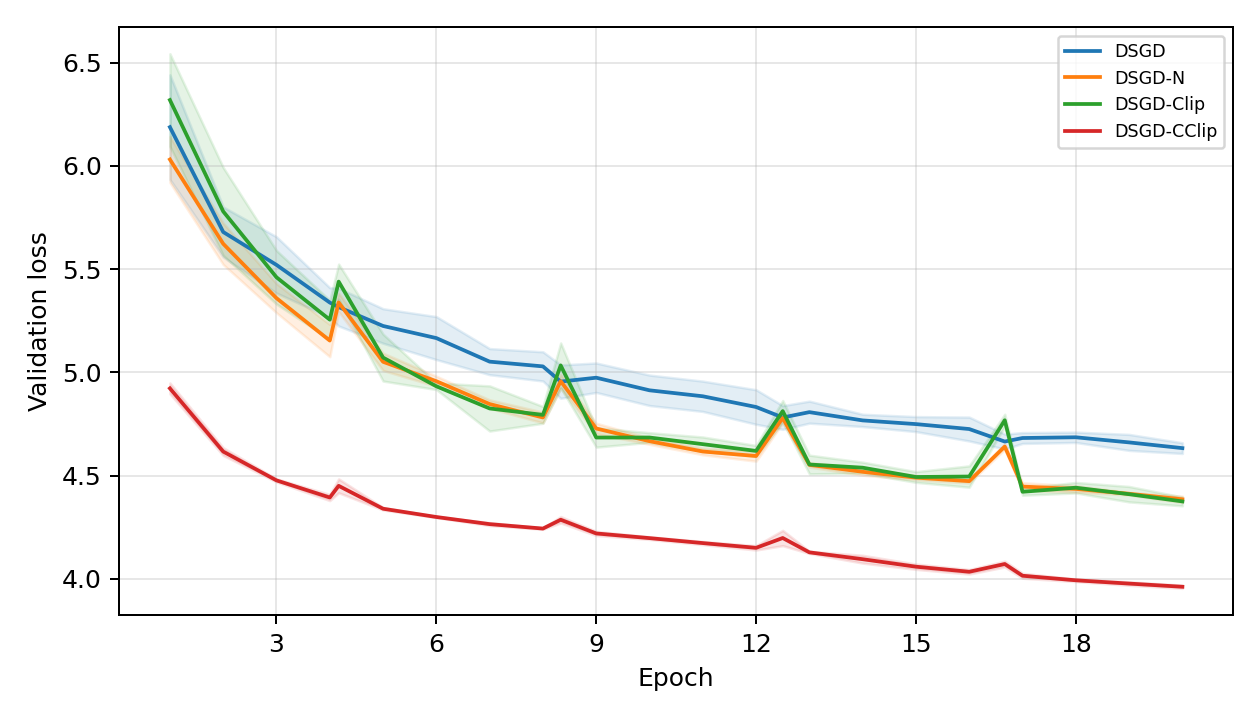}
&
\includegraphics[width=0.45\linewidth]{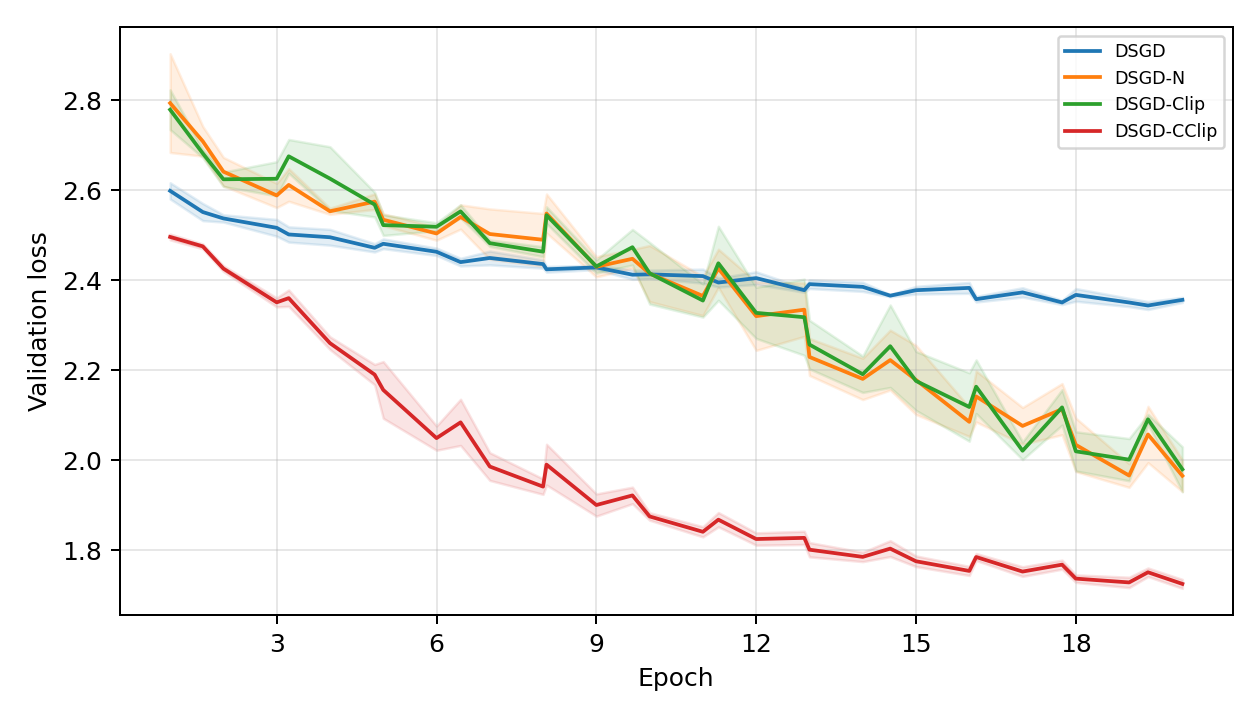}
\end{tabular}
\caption{Validation loss versus training epoch for decentralized transformer training. Left: ``Multi30k''; right: ``Tiny Shakespeare'' dataset.}
\label{fig:ring-curves}
\end{figure}

\section{On the Heterogeneity Condition}\label{sup:het}

As discussed in the main body, even the vanilla \dsgd method requires a heterogeneity condition to ensure convergence under non-convex costs. Assumption \ref{asmpt:bdd-het} was previously used for vanilla \dsgd in works including \citep{vlaski-dsgd-nonconv-1,vlaski-dsgd-nonconv-2,armacki2026dsgd}. However, convergence of vanilla \dsgd has been shown under a slightly weaker notion of heterogeneity, namely the \emph{average heterogeneity}, i.e.,
\begin{equation}\label{eq:avg-het}
    \frac{1}{n}\sum_{i \in [n]}\|\nabla f_i(x)\|^2 \leq A^2 + B^2\|\nabla f(x)\|^2,
\end{equation} see, e.g., \citep{NIPS2017_f7552665,pmlr-v119-koloskova20a}. In our work, control of gradients of each agent is needed to show that gradients of agents fall below the clipping threshold with high probability. This is achieved by using the heterogeneity condition in Assumption \ref{asmpt:bdd-het} and controlling the global gradient and the consensus gap, i.e., through the bound
\begin{equation*}
    \|\nabla f_i(\xit)\| \leq L\|\xit - \oxt\| + \|\nabla f_i(\oxt)\| \leq L\|\xit - \oxt\| + A + B\|\nabla f(\oxt)\|. 
\end{equation*} While it is possible to incorporate the averaged heterogeneity condition in our work through the bound
\begin{equation*}
    \max_{i \in [n]}\|\nabla f_i(x)\|^2 \leq \sum_{i \in [n]}\|\nabla f_i(x)\|^2 \stackrel{\eqref{eq:avg-het}}{\leq} nA^2 + nB^2\|\nabla f(x)\|^2,
\end{equation*} we note that it incurs a factor of $\sqrt{n}$ in both $A$ and $B$ compared to Assumption \ref{asmpt:bdd-het}, due to the loose bound in the first inequality. Seeing that both $A$ and $B$ in our results appear as conditions on the step-size, this loose bound would not result in the loss of linear speed-up and would only affect the transient time needed to achieve the global network independent rate $\wbigO\Big((nT)^{\frac{2(1-p)}{3p-2}}\Big)$, resulting in slightly larger transient time.

\section{On the Transient Time}\label{sup:tran-time}

In this section we provide details on the transient time resulting from our bounds. Noting that the bounds in Theorems \ref{thm:non-cvx-hp} and \ref{thm:non-cvx-ie} are already worst-case with respect to $n$ and $1- \lambda$, we can directly inspect the bound, which is given by
\begin{align}\label{eq:tran-t-non-conv-1}
    \widetilde{\bigO}\Bigg(\max\Bigg\{(nT)^{\frac{2(1-p)}{3p-2}}, \frac{n^{\frac{1}{3p-2}}}{T^{\frac{3(p-1)}{3p-2}}}, \frac{n^{\frac{-p}{6p-4}}T^{\frac{1-2p}{3p-2}}}{(1-\lambda)},\frac{n^{\frac{2(p-1)}{3(3p-2)}}}{(1-\lambda)^{\frac{4}{3}}T^{\frac{7p-4}{3(3p-2)}}}, \nonumber \\ \frac{n^{\frac{2-p}{3(3p-2)}}}{(1-\lambda^2)^{\frac{2}{3}}T^{\frac{7p-6}{3(3p-2)}}}, \frac{T^{-2/3}}{(1-\lambda^2)^{2/3}}, \frac{T^{-1}}{(1-\lambda^2)^2}, \frac{\sqrt{n}T^{-1}}{1-\lambda} \Bigg\}\Bigg).
\end{align} Since our goal is to find the smallest $T$ such that the bound in \eqref{eq:tran-t-non-conv-1} becomes $\wbigO\Big((nT)^{\frac{2(1-p)}{3p-2}}\Big)$, it can be readily verified that this is the case for 
\begin{align}\label{eq:tran-t}
    T = \wbigO\Bigg(\max&\Bigg\{n^{\frac{2p-1}{p-1}},\frac{n^{\frac{3p-4}{2}}}{(1-\lambda)^{3p-2}}, \frac{n^{\frac{8(p-1)}{2+p}}}{(1-\lambda)^{\frac{4(3p-2)}{2+p}}}, \nonumber \\ &\frac{n^{\frac{5p-4}{p}}}{(1-\lambda)^{\frac{2(3p-2)}{p}}}, \frac{n^{3(p-1)}}{(1-\lambda)^{3p-2}}, \frac{n^{\frac{2(p-1)}{p}}}{(1-\lambda)^{\frac{2(3p-2)}{p}}}, \frac{n^{\frac{7p-6}{2p}}}{(1-\lambda)^{\frac{3p-2}{p}}} \Bigg\}\Bigg).
\end{align} We note two things. First, as $p  \rightarrow 1$, the term $n^{\frac{2p-1}{p-1}}$ dominates and the transient time becomes independent of network connectivity, only depending on the number of agents. Second, in the special case when $p = 2$, the transient time in \eqref{eq:tran-t} evaluates to $\wbigO\Big( \frac{n^3}{(1-\lambda)^4}\Big)$, matching the transient time of vanilla \dsgd under bounded variance/sub-Gaussian noise in \citep{pmlr-v119-koloskova20a,unified-refined,armacki2026dsgd}.

\end{document}

%% file: my_macros.tex
\usepackage{nicefrac}
\usepackage{amsmath}
\usepackage{amsthm}
\usepackage{amssymb}
\usepackage{amsfonts}
\usepackage{comment}
\usepackage{mathtools}
\usepackage{algorithm}
\usepackage{graphicx}
\usepackage{xspace}

\DeclareMathOperator*{\argmin}{arg\,min}

\usepackage{xcolor}
\usepackage{enumitem}

\usepackage{bbding}

\newtheorem{theorem}{Theorem}

\newtheorem{assumption}{Assumption}
\newtheorem{lemma}{Lemma}
\newtheorem{remark}{Remark}

\newcommand{\bigO}{\mathcal{O}}
\newcommand{\wbigO}{\widetilde{\bigO}}

\newcommand{\sfo}{$\mathcal{SFO}$\xspace}

\newcommand{\dsgd}{\textbf{\texttt{DSGD}}\xspace}
\newcommand{\dsgdclip}{\textbf{\texttt{DSGD-Clip}}\xspace}
\newcommand{\dsgdcclip}{\textbf{\texttt{DSGD-CClip}}\xspace}
\newcommand{\dsgdn}{\textbf{\texttt{DSGD-N}}\xspace}

\newcommand{\tbgt}{{\widetilde{\mathbf g}^{t}}}

\newcommand{\bx}{{\mathbf x}}
\newcommand{\bxt}{{\mathbf x^{t}}}

\newcommand{\bxtp}{{\mathbf x^{t+1}}}
\newcommand{\bxtm}{{\mathbf x^{t-1}}}

\newcommand{\Prob}{\mathbb{P}}

\usepackage{xcolor}

\usepackage{tikz}
\usetikzlibrary{calc,trees,positioning,arrows,chains,shapes.geometric,%
	decorations.pathreplacing,decorations.pathmorphing,shapes,%
	matrix,shapes.symbols}

\tikzstyle{startstop} = [rectangle, draw, rounded corners, align=center, minimum width=3cm, minimum height=1cm,text centered]
\tikzstyle{decision} = [diamond, draw, fill=blue!20, 
text width=4.5em, text badly centered, node distance=3cm, inner sep=0pt]
\tikzstyle{block} = [rectangle, draw, fill=blue!10, align=center, rounded corners, minimum width=3cm, minimum height=1cm]
\tikzstyle{blockcast} = [rectangle, draw, fill=red!10, align=center, rounded corners, minimum width=3cm, minimum height=0.45cm]
\tikzstyle{line} = [draw, -latex']
\tikzstyle{cloud} = [draw, ellipse,fill=red!20, node distance=3cm,
minimum height=2em]

\newcommand{\R}{\mathbb{R}}
\newcommand{\E}{\mathbb{E}}
\newcommand{\N}{\mathbb{N}}

\newcommand{\calN}{\mathcal{N}}

\newcommand{\Ft}{\mathcal{F}_t}

\newcommand{\bW}{\mathbf{W}}

\newcommand{\bJ}{\mathbf{J}}

\newcommand{\ox}{\overline{x}}
\newcommand{\obx}{\overline{\bx}}

\newcommand{\nbft}{\nabla \mathbf{f}^t}
\newcommand{\nbfk}{\nabla \mathbf{f}^k}

\newcommand{\onft}{\overline{\nabla} f^t}

\newcommand{\ozt}{\overline{z}^t}

\newcommand{\xit}{x^t_i}
\newcommand{\xik}{x^k_i}
\newcommand{\xitp}{x^{t+1}_i}

\newcommand{\wbgk}{\widetilde{\mathbf{g}}^k}

\newcommand{\wbgt}{\widetilde{\mathbf{g}}^t}

\newcommand{\git}{g^t_i}

\newcommand{\tgit}{\widetilde{g}^t_i}
\newcommand{\tgik}{\widetilde{g}^k_i}
\newcommand{\tgjt}{\widetilde{g}^t_j}
\newcommand{\tgitb}{\widetilde{g}^t_{i,b}}
\newcommand{\tgitu}{\widetilde{g}^t_{i,u}}
\newcommand{\tgikb}{\widetilde{g}^k_{i,b}}
\newcommand{\tgiku}{\widetilde{g}^k_{i,u}}
\newcommand{\ogt}{\overline{g}^t}

\newcommand{\hgt}{\widehat{g}^t}

\newcommand{\hgtu}{\widehat{g}^t_u}
\newcommand{\hgtb}{\widehat{g}^t_b}
\newcommand{\hgku}{\widehat{g}^k_u}
\newcommand{\hgkb}{\widehat{g}^k_b}
\newcommand{\zit}{z^t_i}

\newcommand{\oxt}{\overline{x}^t}
\newcommand{\oxk}{\overline{x}^k}
\newcommand{\oxtp}{\overline{x}^{t+1}}
\newcommand{\obxt}{\overline{\bx}^t}
\newcommand{\obxtp}{\overline{\bx}^{t+1}}

\newcommand{\fs}{f^\star}

\newcommand{\clip}{\texttt{clip}}

\DeclareMathOperator{\bern}{Bernoulli}